%% file: main.tex
\documentclass[a4paper,12pt,reqno,openany]{amsart} %
\usepackage[verbose]{geometry}
\pdfoutput=1                    %

\input{toinclude}

\usepackage{svg}
\usepackage{subcaption}         %

\hypersetup{
  pdftitle={},
  pdfauthor={Gianmarco Brocchi and Andreas Rosén},
  colorlinks=true,
  citecolor=[rgb]{0.682, 0.133, 0.188}, %
  linkcolor=[rgb]{0,0,0.5},             %
  urlcolor=[rgb]{0,0,0.75},     
  pdfstartview=FitH,
  pdfdisplaydoctitle=true,
  pdflang=en-GB,
  breaklinks=true
}

\usepackage{todonotes}
\providecommand{\inhom}[1]{#1}
\providecommand{\homg}[1]{{#1}^0}

\usepackage[backend=biber,backref=true,style=alphabetic]{biblatex}
\AtEveryBibitem{%
  \clearlist{language}%
}
\usepackage{csquotes}           %
\theoremstyle{plain}
\newtheorem{theorem}{Theorem}[section]  %
\newtheorem{cor}[theorem]{Corollary}    %
\newtheorem{lemma}[theorem]{Lemma}
\newtheorem{prop}[theorem]{Proposition}

\theoremstyle{definition}
\newtheorem{define}[theorem]{Definition}
\theoremstyle{remark}
\newtheorem{remark}[theorem]{Remark}

\numberwithin{equation}{section} %
\allowdisplaybreaks[4] %

\title[quadratic estimates on cylinders]{Global quadratic estimates for degenerate elliptic operators on cylinders}
\author[Brocchi]{Gianmarco Brocchi}
\address{
  Háskóli Íslands \\
  Tæknigarður \\
  Dunhagi 5 \\
  107 Reykjavík \\
  Ísland}
\email{gianmarco@hi.is}
\author[Rosén]{Andreas Rosén}
\address{
  Mathematical Sciences \\
  Chalmers University of Technology and the University of Gothenburg \\
  SE-412 96 Göteborg \\
  Sweden}
\email{andreas.rosen@chalmers.se}

\begin{document}
\begin{abstract}                %
  On $d$-dimensional cylinders $\mathcal{C}= \mathbb{R}^k\times N$, with a closed manifold $N$ as base and large scale dimension $k\in[1,d)$,
  we prove quadratic estimates in weighted $L^2$ space for Dirac operators perturbed by bounded, measurable and accretive coefficients.
  This gives in particular homogeneous Kato square root estimates on $\mathcal{C}$ for Riesz transforms associated with second order divergence form elliptic operators,
  having measurable coefficients with degeneracy governed by a Muckenhoupt $A_2$ weight.
  By localisation and scaling, it also yields local quadratic estimates for perturbed Dirac operators on general manifolds with locally thin cylindrical geometry, and possibly with zero injectivity radius.
\end{abstract}

\maketitle
\tableofcontents%

\section{Introduction}

\input{sections/intro}

\section{Preliminaries}\label{sec:prelims}

\input{sections/prelims}
\subsection{Weighted estimates for unperturbed operators}\label{sec:unperturbed}

\input{sections/unperturbed}
\section{Global quadratic estimates on cylinders}\label{sec:globalQE}

\input{sections/globalQE}

\subsection{Carleson estimate}\label{subsec:carleson_estimate}

\input{sections/Carleson_estimate}
\section{Quadratic estimates on locally cylindrical manifolds}\label{sec:localQE_M}

\input{sections/zero_inj}
\appendix
\section{Poincaré estimates}\label{appx:poincare}

\input{sections/poincare-estimates}
\section{Singular integrals}\label{appx:singular_integrals}

\input{sections/singular_integrals}

\section*{Acknowledgements}
G.B. was supported by the Knut and Alice Wallenberg foundation,
KAW grant 2020.0262 %
postdoctoral program in Mathematics for researchers from outside Sweden.
A.R. was supported by Grant 2022-03996 from the Swedish research council, VR.

\printbibliography%

\end{document}

%% file: toinclude.tex
\usepackage[english]{babel}
\usepackage[utf8]{inputenc}

\usepackage{mathtools}
\usepackage{amsmath, amssymb, amsthm}
\usepackage{enumitem}
\usepackage{esint}              %

\usepackage{stmaryrd}               %
\usepackage{mathrsfs}           %
\usepackage{cases}              %

\usepackage{graphicx}           %

\usepackage{tikz}
\usetikzlibrary{decorations, patterns, cd} %
\usepackage[hang, small, bf]{caption}

\usepackage[pdftex, colorlinks=true, linkcolor=blue, citecolor=magenta, urlcolor=blue]{hyperref}
\usepackage[noabbrev,capitalize]{cleveref} %
\crefformat{section}{\S#2#1#3} %
\crefformat{subsection}{\S#2#1#3}
\crefformat{subsubsection}{\S#2#1#3}

\providecommand{\romannum}[1]{%
  \textup{\uppercase\expandafter{\romannumeral#1}}%
}

\newcommand{\D}{ \, \mathrm{d}}
\providecommand{\1}{ \mathbf{1}} %
\newcommand{\supp}{ \operatorname{supp}} %

\newcommand{\F}{\mathcal{F}}

\DeclareMathAlphabet\EuScript{U}{eus}{m}{n}
\def\Yint#1{\mathchoice
    {\YYint\displaystyle\textstyle{#1}}%
    {\YYint\textstyle\scriptstyle{#1}}%
    {\YYint\scriptstyle\scriptscriptstyle{#1}}%
    {\YYint\scriptscriptstyle\scriptscriptstyle{#1}}%
      \!\iint}
\def\YYint#1#2#3{{\setbox0=\hbox{$#1{#2#3}{\iint}$}
    \vcenter{\hbox{$#2#3$}}\kern-.51\wd0}}
\def\longdash{{-}\mkern-3.5mu{-}} 

\def\fiint{\Yint\longdash}

\providecommand{\Real}{\ensuremath \Re\mathfrak{e}} %
\providecommand{\R}{\ensuremath \mathbb{R}}

\DeclareMathAlphabet\mathcalboondox{U}{BOONDOX-calo}{m}{n} %

\providecommand{\transpose}{\ensuremath \intercal} %

\providecommand{\dom}{\ensuremath \mathsf{dom}}
\providecommand{\nul}{\ensuremath \mathsf{ker}}
\providecommand{\ran}{\ensuremath \mathsf{im}}
\newcommand{\clos}[1]{\overline{#1}}

\DeclareMathAlphabet{\mathcalboondox}{U}{BOONDOX-calo}{m}{n}
\SetMathAlphabet{\mathcalboondox}{bold}{U}{BOONDOX-calo}{b}{n}
\DeclareMathAlphabet{\mathbcalboondox}{U}{BOONDOX-calo}{b}{n}

%% file: sections/intro.tex
The classical $L^p(\mathbb{R}^d)$ boundedness, $1<p<\infty$, of the Riesz transform $\nabla(-\Delta)^{-1/2}$
has a deep and far-reaching extension in the celebrated Kato square root estimate
\begin{equation*}
  \lVert \sqrt{-\mathrm{div} A \nabla} u \rVert_{L^2(\mathbb{R}^d)} \eqsim \lVert \nabla u \rVert_{L^2(\mathbb{R}^d)},
\end{equation*}
proved in~\cite{SolKato2002} for general bounded, measurable, and accretive coefficient matrices $A=A(x)$.
The singular integral operators appearing here are beyond classical Calder\'on--Zygmund operators,
as they may lack pointwise kernel estimates and may fail to be bounded for general $p \in (1,\infty)$.
Boundedness on $L^2$ is instead handled through functional calculus and quadratic estimates,
and here we follow the first order framework from~\cite{AKMc}, 
using a weighted Dirac type operator $D$ perturbed by a multiplier $B$, related to the coefficient matrix $A$.

In~\cite[Theorem 1.1]{AMR},
the quadratic estimate
\begin{equation}\label{eq:scalar_weighted_QE}
  \int_0^\infty \lVert t \inhom{D}\inhom{B}(I + (t \inhom{D}\inhom{B})^2)^{-1} u \rVert^2_{L^2(M,w)} \frac{\mathrm{d}t}{t} \lesssim \lVert u \rVert^2_{L^2(M,w)}
\end{equation}
was proved for certain first-order differential operators
$\inhom{D}\inhom{B}$ with bounded, accretive coefficients $\inhom{B}$ %
on complete Riemannian manifolds $M$ with lower bounds on Ricci curvature
and injectivity radius.
The $L^2$-norm is weighted by a local Muckenhoupt weight $w$, see \cref{subsec:weights}.
This estimate is a state-of-the-art extension of the Kato square root estimate, 
generalising $\mathbb{R}^d$ to a Riemannian manifold $M$
and $L^2(\mathbb{R}^d,\mathrm{d}x)$ to the weighted space $L^2(M,w)$.
By a result of Anderson and Cheeger \cite{AndersonCheeger92},
the geometric hypothesis on $M$ entails
that, uniformly on some fixed scale $t_0 >0$,
the manifold $M$ is locally Lipschitz equivalent
to the flat Euclidean space $\mathbb{R}^d$.
This equivalence is used in~\cite{AMR} for the proof of the quadratic estimate~\eqref{eq:scalar_weighted_QE}.

The assumption that the injectivity radius is positive excludes many manifolds, 
such as those with a thin cylindrical geometry at infinity, and it is desirable to remove this assumption. 
As a motivating application, consider the matrix-weighted Kato square root estimate
\begin{equation}\label{eq:matrix_weighted_Kato}
    \lVert \sqrt{ - (1/a) \mathrm{div} A \nabla  } u \rVert_{L^2(\mathbb{R}^d, \mu)} \eqsim \lVert \nabla u \rVert_{L^2(\mathbb{R}^d, W)}
\end{equation}
where $\mu$ and $W$ are a scalar and a matrix weight, and the coefficients $W^{-1/2} A W^{-1/2}$ and $a/\mu$ are uniformly bounded and accretive.
It was shown in the companion paper~\cite[Corollary 2.6]{BrR1} to the present one,
that (an inhomogeneous version of) the matrix-degenerate estimate~\eqref{eq:matrix_weighted_Kato} reduces to the scalar weighted estimate~\eqref{eq:scalar_weighted_QE}
on a manifold $M$ whose metric is described by the matrix weight $W$.
Natural examples of such manifolds are not locally Lipschitz equivalent to $\mathbb{R}^d$,
and the proof in~\cite{AMR} does not apply.
In particular, the injectivity radius of such $M$ may be zero.
To illustrate this, consider the matrix weight $W$ on $\mathbb{R}^2 \setminus \{(0,0)\}$
from \cite[Example 2.8]{BrR1} that describes a
degeneracy of coefficients $A$.
The unit discs in the metric described by $W$ are
shown in \cref{fig:ellipses2}.

\begin{figure}[th]
  \includegraphics[width=8cm]{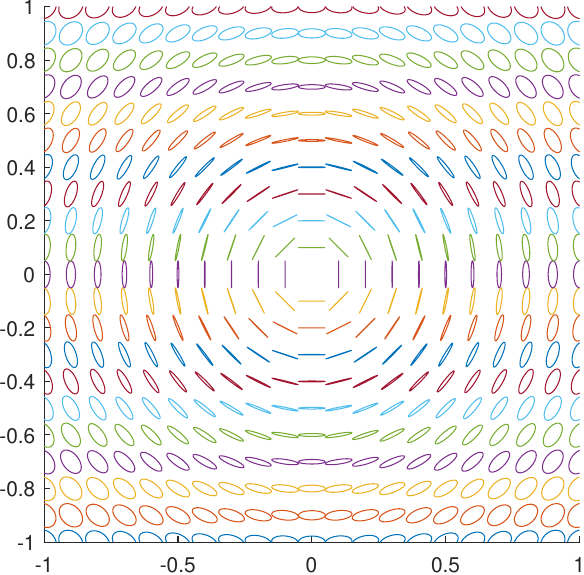}
  \caption{Geodesic disks in the metric given by $W$ are ellipses with increasing eccentricity.}\label{fig:ellipses2}
\end{figure}

In this example, the matrix weight $W$ was chosen so that the manifold $M$ described by 
$W$ equals the graph of the function
$(x^2+y^2)^{-1}$, which is what we see in 
\cref{fig:ellipses2} from above.
We note that the infinity of $M$ near the origin is an increasingly thin cylinder and here
the injectivity radius goes to zero.

In the present paper, we extend~\cite{AMR} to manifolds that are locally Lipschitz equivalent to thin cylinders,
which means that geodesic balls $B_M$ on $M$ satisfy %
\begin{equation*}
  B_M(z,t_0) \simeq (0,1)^k \times \varepsilon N
\end{equation*}
for $z\in M, t_0\eqsim 1$, 
where $(0,1)^k$ denotes the unit cube in $\mathbb{R}^{k}$,
$N$ is a connected, closed manifold (= compact, without boundary) of fixed dimensions $n$ 
and the radius $\varepsilon=\varepsilon(z)\in(0,1)$ is allowed to have $\inf_{z\in M} \varepsilon(z)=0$.
The two main results of our paper are:
\begin{itemize}
\item \cref{thm:QE_zero_inj_radius}, which proves quadratic estimates for locally cylindrical manifolds.
  The quadratic estimates in \cref{thm:QE_zero_inj_radius} are local in the sense that the non-trivial part of the integral to bound is $\int_0^1$, since the operator ${D}_M$ featuring in~\eqref{eq:QE_M} in has lower order terms causing ${D}_M {B}_M$ to have closed range.

\item \cref{thm:global_QE_cylinder}, which contains the local and main part of the estimate \eqref{eq:QE_M},
  and proves global quadratic estimates of cylinders $\mathcal{C} \coloneqq \mathbb{R}^k \times N$.
  The quadratic estimates here are indeed global, since the operator $D^0$ featuring there is a homogeneous first order operator but still all scales $\int_0^\infty$ on the cylinder are controlled.
  Below in this introduction, we outline how rescaling yields the local estimates in \cref{thm:global_QE_cylinder}.
\end{itemize}

\begin{figure}[bh]
  \centering
  \includesvg[scale=1.2]{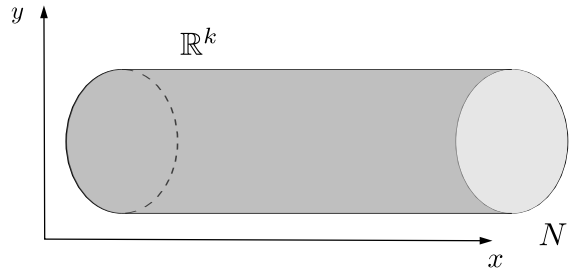}
  \caption{The cylinder $\mathbb{R}^k \times N$.
    The horizontal variable $x$ is along $\mathbb{R}^k$,
  while the vertical variable $y$ parametrizes the compact base manifold $N$.}\label{fig:cylinder}
\end{figure}

These results are new, and in particular give new results for matrix weighted quadratic estimates and Kato square root estimates.
Indeed,~\cite[Theorem 2.4]{BrR1}
and~\cite[Corollary 2.6]{BrR1}
continue to hold if the hypothesis made there that
``$M$ has Ricci curvature bounded from below and positive injectivity radius''
is weakened to 
``$M$ is a locally cylindrical manifold  satisfying hypotheses (H1)–(H3)'' as in
\cref{thm:QE_zero_inj_radius} in the present paper.

We remark that in the unweighted case $w \equiv 1$, the quadratic estimate~\eqref{eq:scalar_weighted_QE} for 
two-dimensional manifolds with local cylinder geometry $[0,1] \times \varepsilon \mathbb{S}^1$,
follows from~\cite[Theorem 6.2]{BandaraMcIntoshGBG}.
More generally this theorem applies whenever the tangent bundle $T N$
has a global frame of parallel vector fields.  
In the present paper,
we do not assume that $M$ has generalised bounded geometry as in~\cite[Definition 2.5]{BandaraMcIntoshGBG}.
In particular, we do not assume the existence of a global frame for $T N$.  
Our extension of~\cite{AMR} in this work applies to a large class of manifolds with zero injectivity radius.
It was already noted in \cite{Lashi_rough}
that $\mathrm{inj}(M) > 0$ is not necessary for proving the Kato square root estimate.

Next, remarks on homogeneity and scaling are in order.
The quadratic estimate \eqref{eq:scalar_weighted_QE}
uses an inhomogeneous first order differential operator $\inhom{D}$
with zero-order terms, so that $\inhom{D}$ has closed range.
As a consequence, the estimate for large scales $t > t_0$ in \eqref{eq:scalar_weighted_QE} is 
trivial, while
the estimate for small scales $t \in (0,t_0)$ follows from that for %
the homogeneous principal part $\homg{D}$ of $\inhom{D}$,
see \cite[Lemma 2.4]{AMR} and \cref{subsec:proof_zero_inj}.
The strategy in \cite{AMR} is to localise the estimate~\eqref{eq:scalar_weighted_QE} to balls $B_M(z,t_0)$
and pull it back to $\mathbb{R}^d$.
Similarly, in this paper we pull back the estimate to the thin cylinders $(0,1)^k \times \varepsilon N$.
By rescaling $(0,1)^k \times \varepsilon N$ to a long cylinder $\frac{1}{\varepsilon} (0,1)^k \times N$,
the local estimate reduces to proving 
\begin{equation}\label{eq:QE_on_rescaled_cylinder}
  \int_0^{t_0/\varepsilon} \lVert s \homg{D}\inhom{B}(I + s^2(\homg{D}\inhom{B})^2)^{-1} u \rVert^2_{L^2(\mathbb{R}^{k}\times N, w)} \frac{\mathrm{d}s}{s} \lesssim \lVert u \rVert^2_{L^2(\mathbb{R}^{k}\times N, w)},
\end{equation}
for the corresponding class of weights $w$ on 
$\mathbb{R}^{k}\times N$.
Note that this estimate must be uniform in $\varepsilon > 0$,
so the estimate for scales $s < t_0/\varepsilon$ is no easier than the one for all $s < \infty$.
Note also that the estimate for the large scales in \eqref{eq:QE_on_rescaled_cylinder}
is no longer trivial, since we have $\homg{D}$ and not $\inhom{D}$
as our self-adjoint differential operator.
Rather it is now the estimate for small scales which is known,
since that follows from \cite[Theorem 1.1]{AMR} upon replacing $M$ by the cylinder $\mathcal{C} \coloneqq \mathbb{R}^{k}\times N$, since this manifold has lower bounds on Ricci curvature and injectivity radius.

At the technical level, a main novelty of the present paper is in \cref{lemma:half_principal_part}
and \cref{prop:principal_part}: a new way to avoid dealing with
the averages of tangent vector fields over large dyadic ``cubes'' in the cylinder $\mathcal{C}$ in the proof of Theorem~\ref{thm:global_QE_cylinder},
which are only well-defined (coordinate-wise)
if there exists a global frame for $T N$.
Another important novelty in the proof is that the standard sectorial decomposition used in the Carleson measure estimate does not apply here. Instead, we are forced to use the refined sectorial 
decomposition from~\cite{Andreas_localTb}.
See \cref{subsec:carleson_estimate}.

Our Theorem~\ref{thm:global_QE_cylinder} addresses a question posed by Alan McIntosh concerning the validity of homogeneous Kato square root estimates on manifolds.
Previously, homogeneous Kato square root estimates
and homogeneous quadratic estimates have only been known for $\mathbb{R}^d$ and for compact manifolds.
Our main result \cref{thm:global_QE_cylinder}
implies the following.

\begin{cor}
  Let $w\in A_2(\mathcal{C})$ be a weight on the cylinder $\mathcal{C} \coloneqq \mathbb{R}^{k} \times N$, with a closed manifold $N$ as base, as in Section~\ref{sec:prelims}.
  Then the (weighted) homogeneous Kato square root estimate
\begin{equation*}
  \lVert \sqrt{-w^{-1}\mathrm{div}w A \nabla} u \rVert_{L^2(\mathcal{C},w)} \eqsim \lVert \nabla u \rVert_{L^2(\mathcal{C},w)}
\end{equation*}
holds.
\end{cor}

For the well known derivation, 
see for example the discussion below \cite[Proposition 3.20]{ARR}.
Note that the inhomogeneous Kato square root estimate
\begin{equation*}
  \lVert \sqrt{-w^{-1}\mathrm{div}w A \nabla} v \rVert_{L^2(\mathcal{C},w)}+\lVert v \rVert_{L^2(\mathcal{C},w)} \eqsim \lVert \nabla v \rVert_{L^2(\mathcal{C},w)}
  +\lVert v \rVert_{L^2(\mathcal{C},w)}
\end{equation*}
is an immediate consequence of the results 
in \cite{ARR}, since it only requires local 
quadratic estimates.
We remark that on flat space, with $\mathcal{C}$ replaced by $\R^d$, the homogeneous estimates follow directly from the inhomogeneous estimates, since $\R^d$ is scale invariant. To see this, set $v(x)=u(\lambda x)$ and let $\lambda \to\infty$, noting that the hypotheses on $A$ and $w$ are scale invariant.
The cylinder $\mathcal{C}$, however, is not 
scale invariant, and we cannot derive \cref{thm:global_QE_cylinder} from 
\cite[Theorem 3.3]{ARR}.

\subsubsection*{Structure of the paper}
In \cref{sec:prelims} we recall definitions of weights on space of homogeneous type and
present an extension for Muckenhoupt weights.
We also collect, in \cref{sec:unperturbed}, the weighted estimates of ``unperturbed'' operators on the cylinder $\mathcal{C}$,
that is, operators not involving the rough coefficients $\inhom{B}$ from \eqref{eq:QE_on_rescaled_cylinder}.

The weighted quadratic estimate \eqref{eq:QE_on_rescaled_cylinder} on the cylinder $\mathcal{C}$ (\cref{thm:global_QE_cylinder})
is proved in \cref{sec:globalQE}.
Applications to weighted quadratic estimates for locally cylindrical manifolds (\cref{thm:QE_zero_inj_radius})
are stated and proved in \cref{sec:localQE_M}.
This class (\cref{def:locally_cylindrical_manifolds}) includes manifolds with zero injectivity radius, as the one from \cite[Example 2.8]{BrR1} discussed above.

The paper contains two appendices: \cref{appx:poincare} with the proofs of Poincaré estimates on the cylinder,
and \cref{appx:singular_integrals} with proofs of singular integral estimates of ``unperturbed'' operators from \cref{sec:unperturbed}.

%% file: sections/prelims.tex
\subsection{Space and setup}\label{subsec:space_setup}
Let $N$ be a closed, connected smooth $n$-dimensional Riemannian manifold, %
with Riemannian metric $g_N$. 
In this article we will only consider connected manifolds.
The cylinder $\mathcal{C} \coloneqq \mathbb{R}^{k} \times N$ is a smooth, complete $d$-dimensional Riemannian manifold, with $d = k + n$.
We endow $\mathcal{C}$ with the product metric $\mathrm{d}(\cdot,\cdot)$ 
and the Riemannian measure $\mathrm{d}\mu \coloneqq \mathrm{d}x\mathrm{d}y$,
where $\mathrm{d}x$ is the standard Lebesgue measure on $\mathbb{R}^{k}$
and $\mathrm{d}y$ is the Riemannian measure on $N$.
The tangent bundle is $T\mathcal{C} \simeq \mathbb{R}^{k} \oplus T N$.
Given a vector field $u$ in $T\mathcal{C}$,
we will write $u$ as $\big[ u_{\mathbb{R}},  u_N \big]^\transpose$, when we want to distinguish between the horizontal and vertical components in $\mathbb{R}^k$ and $T N$.
We will refer to vector field $[u_{\mathbb{R}}, 0]^\transpose$ as being ``horizontal'', and vector fields $[0,u_N]^\transpose$ as ``vertical'', as we did in \cref{fig:cylinder}.  
We denote by $\Delta = \mathrm{div}\nabla$ the unweighted Laplace--Beltrami operator acting on scalar functions on $\mathcal{C}$.

\subsubsection{Space of Homogeneous Type and Dyadic system}
The measure $\mu$ on $\mathcal{C}$ is doubling, in the sense that there exists a finite constant $c_\mu \ge 1$ such that
\begin{equation*}
  \mu\big( B(z,2r) \big) \le c_\mu \, \mu \big( B(z,r) \big)
\end{equation*}
for any $z \in \mathcal{C}$ and any $r >0$, where the geodesic balls $B(z,r)$ are taken with respect to the metric $\mathrm{d}(\cdot,\cdot)$.
Then the cylinder $\mathcal{C}$, viewed as a measure metric space $(\mathcal{C}, \mathrm{d}, \mu)$, %
is a space of homogeneous type (SHT) in the sense of Coifman and Weiss~\cite{CoifWeiss}. 
Being a space of homogeneous type, by a result of Christ~\cite[Theorem 11]{zbMATH00010412}, refined in~\cite{HytKairema},
there exists a system of dyadic cubes $\mathcalboondox{D}$ on $\mathcal{C}$.
For $j \in \mathbb{Z}$, we denote by $\mathcalboondox{D}_{2^{j}}$ the family of dyadic cubes $Q$ of side length $\ell(Q) = 2^{j}$.
For $t>0$, we set $\mathcalboondox{D}_t$ to be $\mathcalboondox{D}_{2^{j}}$
where $j = \mathrm{ceiling}(\log_2t)$, so $2^{j-1} < t \le 2^{j}$.
  We shall however require only large dyadic cubes ($t \ge 1$),
  in which case we use cubes $Q = P \times N$,
  where $P$ is a standard dyadic cube in $\mathcalboondox{D}_{2^j}(\mathbb{R}^{k})$ for $j \in \mathbb{N} \cup \{0\}$.

  \subsubsection{Notations and structural constants}\label{subsec:constants}
  For two quantities $X, Y \ge 0$,
  the expression $X \lesssim Y$ means that there exists a finite, positive constant $C$ such that $X \le C Y$.
  The expression $X \gtrsim Y$ means that $Y \lesssim X$.
  When both expressions hold at the same time, with possibly different constants, we will write $X \eqsim Y$.  
  In particular, we shall omit the dependence on any structural constant, such as the dimensions of the cylinder and the doubling constants depending on the measure.
  
We will consider the weighted measure $\mathrm{d}w \coloneqq w \mathrm{d}\mu$, where $w$ is a scalar Muckenhoupt weight
in $A_2(\mathcal{C})$ with respect the underlying measure $\mathrm{d}\mu$ on $\mathcal{C}$.

\subsection{Muckenhoupt weights}\label{subsec:weights}
Let $(M,g)$ be a complete Riemannian manifold with measure $\mu$.
We include definitions and properties for general $p > 1$,
although only $p=2$ is used in the rest of the article.
Let $B = B(z,r)$ be a geodesic ball of radius $r>0$ centred at $z$.
If $\lvert B \rvert$ denotes the Riemannian measure of a ball $B$,
the average of a scalar function $w$ over $B$ is $\langle w \rangle_B \coloneqq \fint_B w \D{\mu} \coloneqq \lvert B \rvert^{-1} \int_B u \D{\mu}$.
\begin{define}[Muckenhoupt $A_p^R$ weights]\label{def:A_p^R}               %
  Let $R>0$ be fixed, and $p \in (1,\infty)$.
  A scalar weight $w \colon M \to [0,\infty]$ belongs to the Muckenhoupt class $A_p^R(M)$,
  with respect to the Riemannian measure $\mathrm{d}\mu$, if
  \begin{equation}\label{eq:A_p^R_condition}
    [w]_{A_p^R} \coloneqq \sup_{\substack{z_0 \in M \\ r < R}} \Big(\fint_{B(z_0,r)} w(z) \D{\mu}(z) \Big) \Big(\fint_{B(z_0,r)} w(z)^{-\frac1{p-1}} \D{\mu}(z) \Big)^{p-1} < \infty .
  \end{equation}
  We say that a weight $w \in A_p(M)$ if
  $[w]_{A_p} \coloneqq \sup_{R > 0} [w]_{A_p^R} $ is finite.
\end{define}

Geodesic balls can be replaced by cubes.
In the case of the cylinder $\mathcal{C}$
we take the supremum over all cubes $Q \subset \mathcal{C}$.
In particular, we require condition~\eqref{eq:A_p^R_condition} for large cubes $Q = P \times N$,
with $P \in \mathcalboondox{D}(\mathbb{R}^k)$, and $\ell(P) \ge 1$.

The weighted measure $w \mathrm{d}\mu$ is doubling when $w$ is a Muckenhoupt weight in $A_p(M)$, see~\cite[(9) in Proposition 7.1.5]{GrafakosClassical}.

\subsubsection{Periodic extensions}
We show that even, periodic extensions of $A_2$ weights are also in $A_2$.
The existence of an extension for Muckenhoupt weights defined on a domain in $\mathbb{R}^d$ has been studied by Holden~\cite{zbMATH00093899}.
Some attempt to generalise Holden's result to metric measure spaces has been done, see~\cite{Kurki-Mudarra22}.
Here we present a simpler extension which is enough for our setting, and that we have not found in the literature.
We start by defining the even, periodic extension of a function.
\begin{define}[Even, periodic extension]\label{def:periodic-extension}
  Let $Q_0 = (0,\ell_0)^d \subset \mathbb{R}^d$, and let $w_0 \colon Q_0 \to (0,\infty)$ be a weight.
  The even extension of $w_0$ on $(-\ell_0,\ell_0)^d$ is defined by reflecting $w_0$ along each coordinate axis:
  \begin{equation*}
    w_e(x_1,x_2,\ldots) \coloneqq w_0(\lvert x_1\rvert, \lvert x_2\rvert, \ldots).
  \end{equation*}
  Then we extend it periodically to the whole space:
  \begin{equation*}
    w(x + 2 \ell_0 k) \coloneqq w_e(x), \quad \forall k \in \mathbb{Z}^d .
  \end{equation*}
  We refer to $w$ as the \emph{even, periodic extension} of $w_0$.
\end{define}
In the following, we will work on cubes, 
but we note that the same construction applies to functions defined on a parallelepiped in $\mathbb{R}^d$,
by reflecting and extending periodically in each direction.

\begin{prop}\label{prop:periodic_even_ext}
  Let $Q_0 = (0,\ell_0)^d \subset \mathbb{R}^d$.
  If $w_0 \in A_2(Q_0)$, meaning that there exists $C_0 < \infty$ such that
  \begin{equation*}
    [w_0]_{A_2(Q_0)} \coloneqq \sup_{Q \subseteq Q_0} \fint_Q w_0 \fint_Q w_0^{-1} \le C_0
  \end{equation*}
  then its even, periodic extension $w$ is in $A_2(\mathbb{R}^d)$ and
  $[w]_{A_2(\mathbb{R}^d)} \le 2^{2d} [w_0]_{A_2(Q_0)}$.
\end{prop}
\begin{proof}[Proof of \cref{prop:periodic_even_ext}]
  Assume that
  \begin{equation*}
    [w_0]_{Q} \coloneqq \fint_Q w_0 \fint_Q w_0^{-1} \le C_0 \quad \forall Q \subseteq Q_0.
  \end{equation*}
  We want to show that there exists $C < \infty$ such that $[w]_{Q} \le C$ for all $Q \subset \mathbb{R}^d$.

  For cubes $Q \subseteq Q_0$ there is nothing to check,
  same for cubes fully contained in a reflected and/or translated copy of $Q_0$.
  Let $a = (a_i)_{i=1,\dots,d}$ be the centre of $Q$, and let $\ell$ be its sidelength. By translation and reflection, without loss of generality
  we can assume that $a_i \in [0,\ell_0/2]$ for all $i$.
  Moreover, we can assume that $\ell\le \ell_0$.
  Indeed, if $\ell>2\ell_0$ then $Q_0\subset Q$, so 
  $\fint_Q w\eqsim \fint_{Q_0} w$ and 
  $\fint_Q w^{-1}\eqsim \fint_{Q_0} w^{-1}$ by periodicity.
  If $\ell_0<\ell\le 2\ell_0$, then
  $\fint_Q w\lesssim \fint_{2Q} w$
  and $\fint_Q w^{-1}\lesssim \fint_{2Q} w^{-1}$.

  Assuming now $a_i \in [0,\ell_0/2]$ and $\ell\le \ell_0$,
  consider the rectangle $R = Q\cap Q_0$, with all sidelengths belonging to $[\ell/2,\ell]$.
  By evenness $\int_Q w\le 2^d \int_R w$.
  Let $\widetilde Q$ be a cube of sidelength $\ell$ such that $R\subseteq \widetilde Q\subseteq Q_0$.
  Then $\fint_Q w\eqsim \fint_R w\lesssim \fint_{\widetilde Q} w$.
  This together with the same estimate for $w^{-1}$ shows that
  \begin{equation*}
    \fint_{Q} w \fint_{Q} w^{-1}\lesssim \fint_{\widetilde Q} w \fint_{\widetilde Q} w^{-1} \le [w]_{A_2(Q_0)}.
  \end{equation*}
\end{proof}
\begin{remark}
  The same proof works for $A_p$-weights, for $p \in (1, \infty)$,
  giving the bound $[w]_{A_p(\mathbb{R}^d)} \le 2^{d p} [w_0]_{A_p(Q_0)}$.
\end{remark}

%% file: sections/unperturbed.tex
In this section we collect weighted estimates needed to prove the quadratic estimate~\eqref{eq:global_QE_cylinder}  %
on the cylinder $\mathcal{C}$.
We consider unperturbed operators, meaning that these estimates are unrelated to the multiplication operator $B$.

\subsubsection{Poincaré inequality}\label{sec:poincare_scalar}
  We will need a weighted Poincaré inequality %
  on $\mathcal{C}$,
  for large dyadic cubes and functions in the weighted Sobolev space
  \begin{equation*}
    H^1(\mathcal{C},w) \coloneqq \big\{ u \in W^{1,1}_{\mathrm{loc}}(\mathcal{C}), u \in L^2(\mathcal{C},w) \, \text{ with } \, \nabla u \in L^2(\mathcal{C};T\mathcal{C},w) \big\}.
  \end{equation*}
  
  \begin{prop}[Scalar weighted Poincaré]\label{lemma:poincare_scalar} 
    Let $Q \in \mathcalboondox{D}_t$ be a dyadic cube on the cylinder
    with $\ell(Q) \eqsim t \ge 1$. Let $w \in A_2(\mathcal{C})$.
    Then for all $u \in H^{1}(\mathcal{C},w)$
    \begin{equation}\label{eq:poincare_scalar}
      \int_{Q} \lvert u - \langle u \rangle_{Q}\rvert^2 \mathrm{d}w \lesssim \int_{Q} \lvert t \nabla u\rvert^2 \mathrm{d}w.
    \end{equation}
    The inequality~\eqref{eq:poincare_scalar}
    holds if the unweighted average $\langle u\rangle_{Q}$ is replaced by the weighted one.
    The implicit constant depends on $[w]_{A_2}$ and structural constants as in \cref{subsec:constants}.
  \end{prop}
  A proof, by reduction to euclidean rectangles, is in \cref{subsec:poincare}.

  We also require a Poincaré type inequality
    \begin{equation}\label{eq:Poincare_vector}
      \int_{Q} \lvert v \rvert^2 \mathrm{d}w \lesssim \int_{Q} \lvert \overline{\nabla} v\rvert^2 \mathrm{d}w,
    \end{equation}
  without an average term, for vertical gradient vector fields $v= \nabla_y u\in H^1(Q; TN, w)$.
  A vector field $v$ on the cylinder $\mathcal{C}$ is a section of the tangent bundle $T\mathcal{C}$.
  Let $\overline{\nabla}$ be the covariant gradient of vector fields.
  We recall that in a local frame $e_j$ on a Riemannian manifold $M$,
  the covariant gradient $\overline{\nabla}$ on $M$ can be written as
  \begin{equation*}%
    \overline{\nabla} v= \sum_j (\nabla_{e_j} v)\otimes e_j^*,
  \end{equation*}
  where $\nabla_{e_i} v$  is the Levi-Civita covariant derivative and $e_j^*$ is the dual frame.
  If $\overline{\nabla} v = 0$, then we say that $v$ is a parallel vector field.
  If $\overline{\nabla}_x v = 0$, then $v$ is constant along $\mathbb{R}^{k}$,
  while if $\overline{\nabla}_y v = 0$ then $v$ is a parallel vector field on $N$. %
  
  In the unweighted case $w=1$, \eqref{eq:Poincare_vector} can be proved from 
  the scalar Poincaré inequality \cref{lemma:poincare_scalar}
  together with a compactness argument to eliminate the 
  average for gradient vector fields.
  This is possible since there are no non-zero parallel
  vertical gradient vector fields. 
  Indeed, $v= \nabla_y u$ and $\overline{\nabla} v=0$ forces $\text{div}_y v=\Delta_y u=0$,
  see \cite[Proposition 11.2.8]{RosenGMA},
  and therefore $v=0$.
  However, it is unclear to us how to control the constant
  in this compactness argument for general $w\in A_2(\mathcal{C})$.
  Therefore we shall instead of this vector Poincaré argument
  use the global singular integral estimate \cref{lemma:ygradest} below.

  \subsubsection{Singular integrals and Square functions estimates}
  We also need weighed estimates for square functions, Riesz, and Beurling transforms.
  These estimates are deduced from known estimates for the heat kernel on the product manifold $\mathbb{R}^k \times N$, for closed $N$.
  Indeed, these operators can be written in terms of derivatives of the heat propagator on the manifold,
  and they are shown to be Calderón--Zygmund operators on a space of homogeneous type,
  for which weighted estimates are known.
  
  \begin{prop}[Square functions estimates]\label{lemma:weighted_square_function_est_cylinder}
    Let $w \in A_2(\mathcal{C})$ and let $\Delta$ be the Laplace--Beltrami operator on the cylinder $\mathcal{C}$.
    Then 
    \begin{equation}\label{eq:weighted_square_function_est_cylinder}
      \int_0^\infty \lVert t (-\Delta)^{1/2} (I - t^2 \Delta)^{-1} u \rVert_{L^2(\mathcal{C},w)}^2 \frac{\mathrm{d}t}{t} \lesssim \lVert u \rVert_{L^2(\mathcal{C},w)}^2
    \end{equation}
    holds for all $u \in L^2(\mathcal{C},w)$.
    The implicit constant depends on $[w]_{A_2}$ and structural constants as in \cref{subsec:constants}.
  \end{prop}
  A proof, through heat kernel estimates and weighted estimates for vector-valued singular integrals on $\mathcal{C}$, is in \cref{subsec:square_function_estimates}.
  
  \begin{prop}[Riesz transform bounds]\label{lemma:Riesz_bounds}
    Let $w \in A_2(\mathcal{C})$, let $\Delta$ be the Laplace--Beltrami operator
    and let $\mathcal{R} \coloneqq \nabla (-\Delta)^{-1/2}$ be the Riesz transform on the cylinder $\mathcal{C}$.
    Then 
    \begin{equation}\label{eq:Riesz_bounds}
      \lVert \nabla (-\Delta)^{-1/2} u \rVert_{L^2(\mathcal{C},w)} \eqsim \lVert u \rVert_{L^2(\mathcal{C},w)}
    \end{equation}
    holds for all $u \in L^2(\mathcal{C},w)$,
    where the implicit constant depends on $[w]_{A_2}$ and structural constants as in \cref{subsec:constants}.
  \end{prop}
  A proof, through heat kernel estimates and weighted estimates for singular integrals on $\mathcal{C}$, is in \cref{subsec:Riesz_transform_bounds}.

  \begin{prop}[Beurling transform bounds]\label{lemma:Beurling_bounds}
    Let $w \in A_2(\mathcal{C})$. Then for all scalar functions $u \in L^2(\mathcal{C},w)$ we have
    \begin{equation}\label{eq:Beurling_bounds}
      \lVert \overline{\nabla} \nabla (-\Delta)^{-1} u \rVert_{L^2(\mathcal{C},w)} \lesssim \lVert u \rVert_{L^2(\mathcal{C},w)}
    \end{equation}
    where $\overline{\nabla}$ is the covariant gradient of vector fields on the cylinder $\mathcal{C}$.
    The implicit constant depends on $[w]_{A_2}$ and structural constants as in \cref{subsec:constants}.
  \end{prop}
  A proof, through heat kernel estimates and weighted estimates for singular integrals on $\mathcal{C}$, is in \cref{subsec:weighted_Riesz_Beurling}.
  Also the proof of the unweighted estimate, that is for $w=1$, uses a non-trivial application of the Bochner identity. This uses (the unweighted case of) the following estimate. 

  \begin{prop}[vector Poincaré]\label{lemma:ygradest}
    Let $w \in A_2(\mathcal{C})$. Then for all scalar functions $u \in L^2(\mathcal{C},w)$ we have
    \begin{equation*}%
      \lVert \nabla_y (-\Delta)^{-1} u \rVert_{L^2(\mathcal{C},w)} \lesssim \lVert u \rVert_{L^2(\mathcal{C},w)}.
    \end{equation*}
    The implicit constant depends on $[w]_{A_2}$ and structural constants as in \cref{subsec:constants}.
  \end{prop}

  A proof is in \cref{subsec:weighted_Riesz_Beurling}.
  Note that it is not possible to control the full gradient $\nabla (-\Delta)^{-1} u$
  in this way, only the vertical gradient $\nabla_y (-\Delta)^{-1} u$
  due to the compactness of the base manifold $N$.

%% file: sections/globalQE.tex
In this section we state and prove our main theorem:
a quadratic estimate for perturbed operators $\homg{D}B$ on cylinders $\mathcal{C} = \mathbb{R}^k \times N$ introduced in \cref{subsec:space_setup}.

\subsection{Main theorem}\label{subsec:preamble_main_thm}
Let $\mathcal{V}$ be the vector bundle over $\mathcal{C} = \mathbb{R}^k \times N$ given by
\begin{equation}\label{eq:vector_bundle_V}
  \mathcal{V} \coloneqq \mathbb{C} \oplus T \mathcal{C}  \simeq \mathbb{C} \oplus \mathbb{R}^{k} \oplus T N .
\end{equation}
Let $w$ be a Muckenhoupt weight in $A_2(\mathcal{C})$.
We will write sections $u$ of $\mathcal{V}$ as $\big[ u_0, u_{\mathcal{C}} \big]^\transpose =$ $\big[ u_0, u_{\mathbb{R}},  u_N \big]^\transpose$
where $u_0 \in L^2(\mathbb{C},w)$, $ u_{\mathbb{R}} \in L^2(\mathbb{R}^k,wI_k)$, and $ u_N \in L^2(T N, w I_n)$.
Consider the homogeneous differential operator acting on $L^2(w)$-sections of $\mathcal{V}$ given by
\begin{equation*}
  \homg{D} \coloneqq
  \begin{bmatrix}
    0 & - \frac{1}{w}\mathrm{div}_{x,y} w \\
    \nabla_{x,y} & 0
  \end{bmatrix}=
  \begin{bmatrix}
    0 & -\frac{1}{w}\mathrm{div}_{x} w & -\frac{1}{w}\mathrm{div}_{y} w \\
    \nabla_{x} & 0 & 0 \\
    \nabla_{y} & 0 & 0
  \end{bmatrix}.
\end{equation*}
Here $\nabla_y$ denotes the covariant gradient on $N$.
  The operators $\nabla$ and $\mathrm{div}$ are defined in the distributional sense,
  as in~\cite[\S 2.2]{AMR},
  so that $\homg{D}$ is a self-adjoint operator in $L^2(\mathcal{V},w)$.
  Let $B \in L^\infty(\mathsf{End}(\mathcal{V}))$ be a bounded multiplication operator from $L^2(\mathcal{V},w)$ to itself.
  We assume that $B$ is pointwise accretive, %
  meaning that
  \begin{equation}\label{eq:kappa_in_accretivity}
    \kappa_B \coloneqq \inf_{\substack{x \in \mathcal{C} \\ v \in \mathcal{V} \setminus \{0\}}} \frac{ \Real \big(B(x) v , v \big)}{\lvert v \rvert^2} > 0.
  \end{equation} 

  \begin{theorem}\label{thm:global_QE_cylinder}
    Let $N$ be a closed, smooth $n$-dimensional Riemannian manifold. %
    Let $\mathcal{C} = \mathbb{R}^k \times N$ be the cylinder over $N$, and 
    let $w \in A_2(\mathcal{C})$ be a Muckenhoupt weight. Let $\homg{D}$ and $B$ be the operators defined above.
    Then the quadratic estimate
    \begin{equation}\label{eq:global_QE_cylinder}
      \int_0^{\infty} \left\lVert t \homg{D} B (I + (t \homg{D} B)^2)^{-1} u \right\rVert^2_{L^2(\mathcal{C},w)} \frac{\mathrm{d}t}{t} \lesssim \lVert u \rVert_{L^2(\mathcal{C},w)}^2
    \end{equation}
    holds for all sections $u$ of the bundle $\mathcal{V}$ in $L^2(\mathcal{C},w)$.
    The implicit constant depends on $[w]_{A_2}, \kappa_B, \lVert B \rVert_{L^\infty}$, and structural constants as in \cref{subsec:constants}.
  \end{theorem}

  The rest of this section is devoted to the proof of \cref{thm:global_QE_cylinder}.
  We build on the proof of quadratic estimates in Euclidean space in~\cite[Theorem 3.3]{ARR}.
  We highlight the novelties that allow us to extend this result to cylinders.

  First, we recall resolvent estimates for $\homg{D} B$ and topological splitting of the space.
  Given an angle $\theta<\pi/2$, we define the closed bisector $\Sigma_\theta$ in the complex plane by
  \begin{equation*}
    \Sigma_\theta \coloneqq \{ z \in \mathbb{C} \colon \lvert \arg(\pm z)\rvert \le \theta \} \cup \{ 0\}.
  \end{equation*}

  \begin{prop}[{\cite[Proposition 3.1]{ARR}}]\label{prop:typeomega}
    Under the above assumptions on $\homg{D}$ and $B$, the following facts hold.  
    \begin{enumerate}[label= (\roman*)]
    \item The operator $\homg{D} B$, with domain $B^{-1}\dom(\homg{D})$, is $\theta$-bisectorial.
      There is an angle $\theta<\pi/2$ such that
      the spectrum $\sigma(\homg{D} B)$ is contained in a bisector
      $\Sigma_\theta$ of angle $\theta$
      and resolvent estimates $\lVert (\lambda - \homg{D} B)^{-1} \rVert \lesssim \mathrm{dist}(\lambda, \Sigma_\mu)^{-1}$ hold for any $\mu \in (\theta, \frac{\pi}2)$ when $\lambda \notin \Sigma_\mu$.
    \item The operator $\homg{D} B$ has range $\ran(\homg{D}B)=\ran(\homg{D})$ and null space $\nul(\homg{D}B)=B^{-1}\nul(\homg{D})$ such that topologically (but not necessarily orthogonally) one has      
      \begin{equation*}
        L^2(\mathcal{V},w) = \clos{\ran(\homg{D}B)} \oplus \nul(\homg{D}B).
      \end{equation*}
    \item The restriction of $\homg{D} B$ to $\clos{\ran(\homg{D})}$
      is a closed, injective operator with dense range in
      $\clos{\ran(\homg{D})}$.  Moreover, the same statements on spectrum and resolvents as in (i) hold.
    \item Statements similar to (i) and (ii) hold for $B \homg{D}$ with $\dom(B \homg{D})=\dom(\homg{D})$,
      since $B\homg{D}$ is the adjoint of $\homg{D} B^*$, or alternatively since $B \homg{D}= B(\homg{D} B)B^{-1}$ on $\ran(B \homg{D})\cap \dom(\homg{D})$,
      where $\ran(B \homg{D}) = B\ran(\homg{D})$ and $B \homg{D}=0$ on the null space $\nul(B \homg{D})\coloneqq \nul(\homg{D})$.
    \end{enumerate}
  \end{prop}
  For the proof we refer to~\cite[Proposition 3.3]{AAMc2010}.
  The resolvent estimates in \cref{prop:typeomega} (i)
  imply that the resolvents $R_t^B \coloneqq (I + i t \homg{D} B)^{-1}$ are uniformly bounded.
  The same holds for the operators
  \begin{equation*}
    P_t^B \coloneqq (I + (t \homg{D} B)^2)^{-1} 
    \quad \text{and} \quad Q_t^B \coloneqq t \homg{D} B \, P_t^B
  \end{equation*}
  since they can be written as linear combinations of $R_t^B$ and  $R_{-t}^B$,
  see~\cite[\S 3.1]{KatoElEscorial2010}.
  Then~\eqref{eq:global_QE_cylinder} is a quadratic estimate for $Q_t^B$,
  which we will also denote by $\psi(t D^0 B)$, where $\psi(\zeta) = \zeta(1+\zeta^2)^{-1}$.
  
  \subsubsection*{Structure and novelties in the proof}
  The standard approach consists in approximating the operator $Q_t^B$ by a paraproduct $\gamma_t E_t$ (principal part),
  where $E_t$ is an averaging operator at scale $t$ (see \cref{def:E_t} below)
  and $\gamma_t$ is the action of $Q_t^B$ on constant functions and vector fields.
  A first novelty here is to consider an averaging operator
  which avoids averaging the vertical component in $T \mathcal{C}$.
  Indeed, there is no natural way of defining this average,
  and we will show that in fact there is no need to average this component.

  The approximation of the principal part $Q_t^B - \gamma_t E_t$ is controlled (\cref{lemma:half_principal_part})
  using off-diagonal estimates and weighted Poincaré inequality on the cylinder.
  This term is then bounded in $L^2(\mathcal{C},w)$ upon introducing a smoothing operator $P_t$ (see \cref{def:P_t}) %
  whose gradient enjoys square function estimates (\cref{lemma:square_function_P_t}).
  The boundedness of the paraproduct $\gamma_t E_t$ is reduced to a Carleson measure estimate via a $T b$ argument,
  where suitable test functions $b$ are constructed using the operator $\homg{D} B$.
  This is proved by adapting the main result of~\cite{Andreas_localTb} to our setting.

  The following is the complement result of~\cite[Lemma 2.4]{AMR}.
  \begin{lemma}\label{lemma:reduction_as_in_AMR}
    Let $\psi(z) = z (1+z^2)^{-1}$, and $\psi_t(z) = \psi(t z)$.
    Let $T$ and $T_0$ be bisectorial operators with same domain on a Hilbert space $\mathcal{H}$
    and such that
    \begin{equation*}
      \int_1^\infty \lVert \psi_t(T_0) u \rVert^2 \frac{\mathrm{d}t}t \lesssim \lVert u \rVert^2 \quad \text{ and } \quad  \int_0^1 \lVert \psi_t(T) u \rVert^2 \frac{\mathrm{d}t}t \lesssim \lVert u \rVert^2
    \end{equation*}
    hold for all $u \in \mathcal{H}$.
    If $T - T_0$ is a bounded linear operator on $\mathcal{H}$, then the quadratic estimate
    \begin{equation*}
      \int_0^\infty \lVert \psi_t(T_0) u \rVert^2 \frac{\mathrm{d}t}t \lesssim \lVert u \rVert^2
    \end{equation*}
    holds for all $u \in \mathcal{H}$.
  \end{lemma}
  \begin{proof}
    Let $R_t(T_0) \coloneqq (I + i t T_0)^{-1}$ be the resolvent of $T_0$ on scale $t$.
    As in~\cite[Lemma 2.4]{AMR}, it is enough to check that the difference of the resolvents is bounded as
    \begin{equation*}
      \lVert R_t (T_0) - R_t(T) \rVert_{\mathcal{H} \to \mathcal{H}} \le \lvert t\rvert \lVert T_0 - T \rVert_{\mathcal{H} \to \mathcal{H}}.
    \end{equation*}
    This implies the bound
    \begin{align*}
      \Big(\int_0^1 \lVert \psi_t(T_0) u \rVert^2 \frac{\mathrm{d}t}t\Big)^{1/2}
                & \le \Big(\int_0^1 \lVert \big(\psi_t(T_0) - \psi_t(T)\big) u \rVert^2  \frac{\mathrm{d}t}t \Big)^{1/2} + \Big(\int_0^1 \lVert \psi_t(T) u \rVert^2 \frac{\mathrm{d}t}t \Big)^{1/2} \\
               & \lesssim \Big(\int_0^1 \lVert u \rVert^2 t \mathrm{d}t \Big)^{1/2} + \lVert u \rVert.
    \end{align*}   
  \end{proof}

  \begin{remark}
    In view of \cref{lemma:reduction_as_in_AMR},
    it is enough to prove quadratic estimates~\eqref{eq:global_QE_cylinder} only for large $t \ge 1$.
    This because we can apply \cref{lemma:reduction_as_in_AMR} with $T = D B$ and $T_0 = \homg{D} B$,
    where $\inhom{D}$ is the inhomogeneous operator considered in~\cite{AMR} and also in~\eqref{eq:inhom_D}.
    Quadratic estimates for $T$ on small scales ($t\le 1$) follow from~\cite[Theorem 1.1]{AMR},
    while the difference $\inhom{D} - \homg{D}$ is a bounded operator on the Hilbert space $\mathcal{H} \coloneqq L^2(\mathcal{C},w)$.
  \end{remark}

  We introduce the operator $P_t$ and prove square function estimates.
\begin{define}[Mollifier $P_t$]\label{def:P_t}
  We define the unperturbed mollifier
  \begin{equation}\label{eq:P_t}
    P_t \coloneqq
    \begin{bmatrix}
      (I - t^2 \Delta_w)^{-1} &  \\
                      & (I - t^2\Delta_H)^{-1}
    \end{bmatrix}
  \end{equation}
  where $\Delta_w \coloneqq (1/w)\mathrm{div}_{x,y} w \nabla_{x,y}$ is the scalar weighted Laplacian on the cylinder $\mathcal{C}$,
  $I_d$ is the $d\times d$ identity matrix,
  and $\Delta_H \coloneqq \nabla \mathrm{div}$ is the unweighted Hodge-Laplacian
  acting on gradient vector fields on $\mathcal{C}$.
  Note that $P_t$ is \emph{not} the operator $P_t^B$ defined above when $B$ is the identity: they coincide only on scalar functions.
\end{define}
\begin{remark}
  We note that $\Delta_H$ will always act on  gradient vector fields, and
  \begin{equation}\label{eq:action_Hodge-Laplacian}
    \Delta_H(\nabla u) = \nabla \mathrm{div} (\nabla u) = \nabla ( \Delta u )
  \end{equation}
  for any scalar function $u$,
  where $\Delta = \mathrm{div}\nabla$ denotes the Laplace--Beltrami operator
  acting on scalar functions.
\end{remark}

 The operator $P_t$ satisfies the following square function estimates.
 \begin{lemma}\label{lemma:square_function_P_t}
   Let $w \in A_2(\mathcal{C})$.
   Then for all $u \in \overline{\mathsf{im}(\homg{D})}$ we have
   \begin{equation*}
     \int_0^\infty \left\lVert t
       \begin{psmallmatrix}
         \nabla & 0 \\
         0 & \overline{\nabla}
       \end{psmallmatrix} P_t
       \begin{bsmallmatrix}
         u_0 \\ u_{\mathcal{C}}
       \end{bsmallmatrix} \right\rVert_{L^2(\mathcal{C},w)}^2 \frac{\mathrm{d}t}t \lesssim
     \left\lVert \begin{bsmallmatrix}
       u_0 \\ u_{\mathcal{C}}
     \end{bsmallmatrix} \right\rVert^2_{L^2(\mathcal{C},w)}.
 \end{equation*}
\end{lemma}
\begin{proof}
  This amounts to showing the two square function estimates
  \begin{align}
    \int_0^\infty \lVert t \nabla (I - t^2\Delta_w)^{-1} f \rVert_{L^2(\mathcal{C},w)}^2 \frac{\mathrm{d}t}t & \lesssim \lVert f \rVert_{L^2(\mathcal{C},w)}^2 , \label{eq:scalar_part_square_function} \\
    \int_0^\infty \lVert t \overline{\nabla} (I - t^2\Delta_H)^{-1} \nabla g \rVert_{L^2(\mathcal{C},w)}^2 \frac{\mathrm{d}t}t & \lesssim \left\lVert \nabla g \right\rVert_{L^2(\mathcal{C},w)}^2 , \label{eq:vector_part_square_function}
  \end{align}
  where $u = [f , \nabla g]^\transpose$, and $\nabla g = [\nabla_x g, \nabla_y g]^\transpose$.
  We start with~\eqref{eq:scalar_part_square_function}, which involves a scalar function $f$.
  Recall that $-\Delta_w = \mathrm{div}_w \nabla$, where $\mathrm{div}_w$ is the adjoint of $\nabla$ with respect to the inner product on $L^2(\mathcal{C},w)$. 
  Using the identity $\|Ax\|= \|\sqrt{A^*A}x\|$, we have
  \begin{align*}
    \lVert t\nabla (I - t^2 \Delta_w)^{-1} f \rVert_{L^2(\mathcal{C},w)}^2 & = \lVert t (-\Delta_w)^{1/2} (I - t^2 \Delta_w)^{-1} f \rVert_{L^2(\mathcal{C},w)}^2 \\
                                          & = \lVert \psi(t\sqrt{-\Delta_w}) f \rVert_{L^2(\mathcal{C},w)}^2
  \end{align*}
  where $\psi(z) = z(1+z^2)^{-1}$ is holomorphic. %
  Since $-\Delta_w$ is self-adjoint, so is $(-\Delta_w)^{1/2}$ and the operator $\psi(t \sqrt{-\Delta_w})$.
  By the spectral theorem for $-\Delta_w$ %
  it holds that
  \begin{equation*}
    \int_0^\infty \lVert t\nabla (I - t^2 \Delta_w)^{-1} f \rVert_{L^2(\mathcal{C},w)}^2 \frac{\D t}{t} = \int_0^\infty \lvert \psi(t) \rvert^2\frac{\D t}{t} \lVert f \rVert_{L^2(\mathcal{C},w)}^2 \lesssim \lVert f \rVert_{L^2(\mathcal{C},w)}^2
  \end{equation*}
  since $\psi$ vanishes at $0$ and decays at infinity.
  For the square function in~\eqref{eq:vector_part_square_function},
  we use that $\nabla g$ can be written as $\mathcal{R} (-\Delta)^{1/2} g$, where $\mathcal{R} \coloneqq \nabla (-\Delta)^{-1/2}$ is the Riesz transform.
  Let $h\coloneqq (-\Delta)^{1/2} g$.
  By the action of the Hodge-Laplacian $\Delta_H$ on gradient vector fields as in~\eqref{eq:action_Hodge-Laplacian},
  we have
  \begin{equation}\label{eq:interaction_P_t_and_Riesz}
    (I - t^2\Delta_H)^{-1} \mathcal{R} h = \mathcal{R} (I - t^2 \Delta)^{-1} h.
  \end{equation}
  After applying the conormal gradient $\overline{\nabla}$, the bound for the Beurling transform in \cref{lemma:Beurling_bounds} gives
  \begin{equation*}
    \lVert \overline{\nabla} (I - t^2\Delta_H)^{-1} \mathcal{R} h \rVert_{L^2(\mathcal{C},w)} \lesssim \lVert \Delta (-\Delta)^{-1/2} (I - t^2\Delta)^{-1} h \rVert_{L^2(\mathcal{C},w)}.
  \end{equation*}
  We conclude via the weighted square function estimate~\eqref{eq:weighted_square_function_est_cylinder}:
  \begin{equation}\label{eq:vector_part_square_function_bound}
    \int_0^\infty \lVert t(-\Delta)^{1/2}(I - t^2\Delta)^{-1} h\rVert_{L^2(\mathcal{C},w)}^2 \frac{\D t}{t} \lesssim \lVert h \rVert_{L^2(\mathcal{C},w)}^2 \eqsim \lVert \nabla g\rVert_{L^2(\mathcal{C},w)}^2 ,
  \end{equation}
  where the last bound is the reverse inequality for the Riesz transform in \eqref{eq:Riesz_bounds}.
\end{proof}
 
A key novelty in the proof of \cref{thm:global_QE_cylinder} is to consider an averaging operator
which avoids averaging the vertical component in $T \mathcal{C}$.
\begin{define}[Averaging operator]\label{def:E_t}
  Given a dyadic cube $R \in \mathcalboondox{D}_t(\mathbb{R}^k)$,
  we define an averaging operator at scale $t\ge 1$, %
  \begin{equation}\label{eq:averaging_operator_E_t}
    E_t
    \begin{bmatrix}
      u_0 \\ u_{\mathbb{R}} \\ u_N
    \end{bmatrix} = \sum_{\substack{Q = R \times N \\ R \in \mathcalboondox{D}_t}}
    \left[\frac{1}{w(Q)} \int_{Q} u_0 \,\mathrm{d}w \,,
    \frac{1}{\mu(Q)} \int_{Q} u_{\mathbb{R}} \,\mathrm{d}\mu \,,
    0 \;
    \right]^\transpose \1_Q .
  \end{equation}
\end{define}
Differently with respect to previous averaging operators $E_t$ in the literature,
we do not average the vertical component $u_N$,
since there is no natural definition of this.
The operator $E_t$ is bounded on $L^2(\mathcal{C},w)$:
the only non trivial bound to check is the unweighted average in weighted norms,
which componentwise is bounded as
\begin{equation*}
  \lVert E_t u_j \rVert_{L^2(\mathcal{C},w)}^2 = \sum_{\substack{Q = R \times N \\ R \in \mathcalboondox{D}_t}} w(Q) \Big(\fint_Q u_j \D{\mu}\Big)^2 \le [w]_{A_2} \sum_{\substack{Q = R \times N \\ R \in \mathcalboondox{D}_t}} \int_Q \lvert u_j\rvert^2 w \D{\mu} \lesssim \lVert u_j \rVert_{L^2(\mathcal{C},w)}^2 ,
\end{equation*}
where the first inequality is H\"older's inequality for $p=2$:
\begin{equation*}%
  \fint_{Q} f \D{\mu} \le [w]_{A_p}^{1/p} \Big( \frac{1}{w(Q)} \int_{Q} \lvert f \rvert^p w \D{\mu} \Big)^{1/p}
\end{equation*}
that holds for any $f \in L^1(Q,\mu)$ and any $p > 1$.

\begin{define}[Principal part]\label{def:principal_part}
  Let $\gamma_t$ be the multiplier in $\mathrm{End}(\mathbb{C}^{1+k};\mathbb{C}^{1+k}\times T N)$
  that at a point $(x,y)\in\mathcal{C}$ is defined as
  \begin{equation*}
    \gamma_t(x,y)
    \begin{bsmallmatrix}
      a \\ b \\ 0
    \end{bsmallmatrix} \coloneqq
    \Big( Q_t^{B}
    \begin{bsmallmatrix}
      a \\ b \\ 0
    \end{bsmallmatrix}
    \Big)(x,y)
    \quad \text{ for } a \in \mathbb{C} , b \in \mathbb{C}^k.
  \end{equation*}
\end{define}
The multiplier $\gamma_t$ is in $L^2_{\mathrm{loc}}$, see \cite[Lemma 3.6]{ARR},
as a consequence of the following off-diagonal estimates.
\begin{lemma}[Off-diagonal estimates]\label{lemma:off-diagonal-est}
  Let $E,F \subseteq \mathcal{C}$ two measurable subsets,
  let $\mathcal{V}$ be the vector bundle over the cylinder defined in \eqref{eq:vector_bundle_V},
  and let $\mathcal{H}$ be the Hilbert space $L^2(\mathcal{V},w)$.
  Then there exists a constant $c_B >0$
  depending only on $\kappa_B$ and $\lVert B \rVert_{L^\infty}$ such that 
  \begin{equation}\label{eq:off-diagonal-estimate}
    \lVert \1_E Q_t^{B} \1_F \rVert_{\mathcal{H} \to \mathcal{H}} \lesssim \exp\Big( - c_B \frac{d(E,F)}{t} \Big)
  \end{equation}
  holds for all $t >0$. The implicit constant depends only on $c_B$.
\end{lemma}
For a proof, we refer to \cite[Lemma 3.1]{AMR} and references therein.
Proofs of off-diagonal estimates go back to \cite[Proposition 5.2]{zbMATH06113005} and \cite[Lemma 5.3 and \S 5]{AndreaAlanAndrew2013}.
The argument goes through for our $\homg{D} B$ operator on the cylinder $\mathcal{C}$.
Inspection of the proof shows that replacing the inhomogeneous $D$ by $D^0$ does not affect the argument.
We also take the opportunity to correct the statement of
\cite[Lemma 3.1]{AMR}: the power $N$ appearing there is 
redundant and should be put to $N=0$.
 
\subsection{Principal part approximation}
We control the difference $Q_t^B - \gamma_t E_t$.
A novelty is the use of %
\cref{lemma:ygradest} for vertical gradient fields on the cylinder.

\begin{prop}\label{lemma:half_principal_part}
  For all $t \ge 1$, the bound
  \begin{equation}\label{eq:half_principal_part}
    \lVert (Q_t^{B} - \gamma_t E_t) u \rVert^2_{L^2(\mathcal{C},w)} \lesssim \lVert t \nabla (u_0, u_{\mathbb{R}}) \rVert^2_{L^2(\mathcal{C},w)}+
    \lVert u_N \rVert^2_{L^2(\mathcal{C},w)}
  \end{equation}
  holds for all Sobolev sections $u=(u_0, u_{\mathbb{R}}, u_N)^\transpose \in H^1(\mathcal{C};\mathcal{V},w)$.
\end{prop}
\begin{proof}[Proof of \cref{lemma:half_principal_part}]
  The proof follows a classical localisation argument with dyadic cubes, see \cite[Lemma 3.6]{KatoElEscorial2010}, and its weighted version in \cite{ARR}, and we therefore omit some details.
  Concretely it amounts to
  \begin{equation*}
    \sum_{Q \in \mathcalboondox{D}_t} \left\lVert Q_t^{B}
      \begin{bmatrix}
       u_0 - \langle u_0\rangle_{w,Q} \\
        u_{\mathbb{R}} - \langle u_{\mathbb{R}} \rangle_{Q} \\
        u_N
      \end{bmatrix}
    \right\rVert_{L^2(Q,w)}^2 \lesssim  \left\lVert
      \begin{bmatrix}
        t\nabla u_0 \\
        t\nabla u_{\mathbb{R}} \\
        u_N
      \end{bmatrix}
    \right\rVert^2_{L^2(\mathcal{C},w)},
  \end{equation*}
  where $\langle u_0\rangle_{w,Q}$ denotes the weighted average of $u_0$ on $Q$.
  To prove \eqref{eq:half_principal_part}, fix a dyadic cube $Q$ with $\ell(Q) \eqsim t$, and decompose space into $A_j \coloneqq 2^{j}Q \setminus 2^{j-1} Q$,   
  for $j \ge 1$, where for large $j$, $2^j Q$ denotes $2^j P \times N$, with $P$ being a dyadic cube in $\mathbb{R}^k$.
  Let $A_0 = Q$,
  so that the ratio $d(Q,A_j)/t \eqsim 2^j$.
  Then off-diagonal estimates in \eqref{eq:off-diagonal-estimate} give
  \begin{equation*}
    \lVert (Q_t^{B} - \gamma_t E_t) u \rVert^2_{L^2(Q,w)} \lesssim_{d,M} \Big( \sum_{j \ge 0} 2^{-j M} \lVert u - E_t u\1_Q \rVert_{L^2(A_j,w)} \Big)^2  %
  \end{equation*}
  for any $M \ge 0$ that will be chosen later.
  When summing over all $Q \in \mathcalboondox{D}_t$, after applying Cauchy--Schwarz in $j$, and enlarging $A_j$ to $2^{j+1}Q$,
  we have
  \begin{align*}
    \sum_{Q \in \mathcalboondox{D}_t} \Big( \sum_{j \ge 0} 2^{-j M} & \lVert u -  E_t u \1_Q \rVert_{L^2(A_j,w)} \Big)^2 \\
                                       & \lesssim \sum_{j\ge 0} \sum_{Q \in \mathcalboondox{D}_t} 2^{-j M} \lVert u - E_t u \1_Q \rVert_{L^2(2^{j+1}Q,w)}^2 .
  \end{align*}
  For the last component $u_N$, this gives the stated
  estimate after summing in $j$ and $Q$.
  For the first two components of $u - E_t u\1_Q$, we use the Poincaré inequality \cref{lemma:poincare_scalar} (componentwise for 
  $u_\R$).
  We obtain
  \begin{equation}\label{eq:expandind_terms_PPA}
    \left\lVert
      \begin{bmatrix}
        u_0 - \langle u_0\rangle_{w,Q} \\
        u_{\mathbb{R}} - \langle u_{\mathbb{R}} \rangle_{Q} 
      \end{bmatrix}
    \right\rVert_{L^2(2^{j+1}Q,w)}^2 \lesssim 2^{2(j+1)} t^2 \left\lVert
      \begin{bmatrix}
        \nabla u_0 \\
        \overline{\nabla} u_{\mathbb{R}} 
      \end{bmatrix}
    \right\rVert^2_{L^2(2^{j+1}Q,w)}.
  \end{equation}
  Note that the averages are over the cube $Q$, while norms are over the enlarged $2^{j+1}Q$.
  This inequality follows from the classical Poincaré inequality via a telescoping argument
  using the averages of the intermediate cubes: %
  \begin{equation*}
    u - \langle u \rangle_Q^w =   \big(u - \langle u \rangle_{2^{j+1}Q}^w\big) + \big(\langle u \rangle_{2^{j+1}Q}^w - \langle u \rangle_{2^{j}Q}^w\big) + \dots + \big(\langle u \rangle_{2Q}^w - \langle u \rangle_Q^w \big).
  \end{equation*}
  Then, for each group in parenthesis,
  the average over the smaller cube $2^{j}Q$ is expanded.
  Using doubling of the weighted measure, Hölder and Poincaré inequality \eqref{eq:poincare_scalar} we have
  \begin{equation*}
    \lvert \langle u \rangle_{2^{k+1}Q}^w - \langle u \rangle_{2^{k}Q}^w \rvert
    \lesssim \Big( \fint_{2^{k+1}Q} \lvert u - \langle u \rangle_{2^{k+1}Q}^w \rvert^2 \D{w} \Big)^{1/2} 
    \lesssim \ell(2^{k+1}Q) \Big( \fint_{2^{k+1}Q} \lvert \nabla u \rvert^2 \D w \Big)^{1/2}. 
  \end{equation*}
  The Poincaré inequality on the enlarged cube $2^{j+1}Q$ follows by summing over $k$.
  The proof is now completed by taking $M$ large enough and summing over all cubes,
  since the family of enlarged cubes $\{2^{j+1}Q, Q \in \mathcalboondox{D}_t\}$ is finite overlapping depending on $j$.
\end{proof}

We now use the smoothing operator $P_t$ from \cref{def:P_t} to write
\begin{equation}\label{eq:insert_P_t}
  Q_t^B - \gamma_t E_t = Q_t^B(I - P_t) + (Q_t^B - \gamma_t E_t)P_t + \gamma_t E_t(P_t - I).
\end{equation}

\begin{prop}\label{prop:principal_part}
  For all $u \in \overline{\mathsf{im}(\homg{D})}$ %
  we have
  \begin{equation*}
    \int_1^\infty \lVert (Q_t^{B} - \gamma_t E_t) u \rVert^2_{L^2(\mathcal{C},w)} \frac{\D{t}}t \lesssim \lVert u \rVert^2_{L^2(\mathcal{C},w)} .
  \end{equation*}
\end{prop}

\begin{proof}%
  We decompose the integrand as in \eqref{eq:insert_P_t} and need to control three terms.
  
Consider first the second term.
  The first two components are controlled by combining \eqref{eq:half_principal_part} 
  with
  the square function estimates for $P_t$ in \cref{lemma:square_function_P_t}.
  For the last component, it remains to show that
  \begin{equation*}
    \int_1^\infty \lVert (P_t u)_N \rVert^2_{L^2(\mathcal{C},w)} \frac{\D{t}}t \lesssim \lVert u \rVert^2_{L^2(\mathcal{C},w)}.
  \end{equation*}
  To see this, we write $(u_\R,u_N)= \mathcal{R}h$
  and use \cref{eq:action_Hodge-Laplacian}
  to get
  \begin{equation*}
    (P_t u)_N= \nabla_y(1-t^2\Delta)^{-1} (-\Delta)^{-1/2}h.
  \end{equation*}
  \cref{{lemma:ygradest}} shows that
  \begin{equation*}
    \lVert \nabla_y(1-t^2\Delta)^{-1} (-\Delta)^{-1/2}h \rVert^2_{L^2(\mathcal{C},w)}
    \lesssim
    \lVert (-\Delta)^{1/2}(1-t^2\Delta)^{-1} h \rVert^2_{L^2(\mathcal{C},w)},
  \end{equation*}
  and so \cref{lemma:weighted_square_function_est_cylinder}
  gives the needed estimate since $t\ge 1$.  

  The first term is controlled by factorising 
  $I - P_t$ as
  \begin{equation*}
    (I - P_t)
    \begin{pmatrix}
      f \\ \mathcal{R}h
    \end{pmatrix} = t
    \begin{pmatrix}
      0 & -\mathrm{div}_w \\
      \nabla & 0
    \end{pmatrix}
    \begin{pmatrix}
      t(-\Delta)^{1/2}(I - t^2\Delta)^{-1} h \\
      t\nabla (I - t^2 \Delta_w)^{-1} f
    \end{pmatrix},
  \end{equation*}
  using 
  \eqref{eq:action_Hodge-Laplacian}.

  Now the uniform boundedness of the operator $Q_t^B t \homg{D} B$ combined with 
  the square function estimates for $f$ and $h$ that have been proved in \eqref{eq:scalar_part_square_function} and in \eqref{eq:vector_part_square_function_bound} respectively, give
  \begin{equation*}
    \int_1^\infty \left\lVert Q_t^{B} (I - P_t) \begin{bsmallmatrix}
      u_0 \\ u_{\mathbb{R}} \\ u_N
    \end{bsmallmatrix} \right\rVert^2_{L^2(\mathcal{C},w)} \frac{\mathrm{d}t}{t} \lesssim \left\lVert \begin{bsmallmatrix}
      u_0 \\ u_{\mathbb{R}} \\ u_N
    \end{bsmallmatrix} \right\rVert^2_{L^2(\mathcal{C},w)} .
\end{equation*}  

The last term in \eqref{eq:insert_P_t} is controlled as in \cite[Lemma 3.14]{ARR}, we highlight the novelties.
We denote by $u_{\mathcal{C}} = [u_{\mathbb{R}}, u_N]^\transpose \in T\mathcal{C}$. We use the Calderón reproducing formula
\begin{equation*}
   \begin{bmatrix}
      u_0 \\ u_{\mathcal{C}}
    \end{bmatrix}
    = c \int_0^\infty \mathbf{Q}_s
    \begin{bmatrix}
      u_0 \\ u_{\mathcal{C}}
    \end{bmatrix} \frac{\mathrm{d}s}s, \quad \text{ with }\quad  \mathbf{Q}_s \coloneqq
    \begin{bmatrix}
      Q_s^w & 0 \\
      0 & Q_s^H
    \end{bmatrix}
\end{equation*}
where
$Q_s^w \coloneqq s(-\Delta_w)^{1/2} (I-s^2\Delta_w)^{-1}$,
$Q_s^H \coloneqq s(-\Delta_H)^{1/2} (I-s^2\Delta_H)^{-1}$,
and $c < \infty$ is a positive constant. 
Note that the Riesz transform intertwines the Hodge-Laplacian and the scalar Laplacian $\Delta$, meaning that
\begin{equation*}
  \Delta_H \mathcal{R} h = \mathcal{R} \Delta h
\end{equation*}
and similarly for functions of $\Delta_H$ and $\Delta$ in their functional calculus.
We denote by $Q_s \coloneqq s(-\Delta)^{1/2} (I-s^2\Delta)^{-1}$.
Since $E_t^2 = E_t$,
and $\gamma_t E_t$ is bounded on $L^2(\mathcal{C},w)$, %
by Schur estimates it is enough to show that
\begin{equation*}
  \lVert E_t(P_t - I) \mathbf{Q}_s \rVert \lesssim (\min\{t/s,s/t\} )^\alpha
\end{equation*}
for some $\alpha > 0$, see \cite[Theorem 4.6.3]{MR2463316}.
Since $E_t$, $(P_t - I)$ and $\mathbf{Q}_s$ are all block-diagonal operators,
we can study their action on the scalar and vector part separately.
The proof for the scalar part is as in the proof of \cite[Lemma 3.14]{ARR}:
using the boundedness of $E_t^w$ on $L^2(\mathcal{C},w)$
and the functional calculus of the other operators,
we bound
\begin{equation}\label{eq:cases_for_Schur_in_trace_term}
  \lVert E_t^w ((I - t^2\Delta_w)^{-1} - I) Q_s^w\rVert \lesssim
  \begin{cases}
    \frac{t}s \lVert Q_t^w (I-(I - s^2\Delta_w)^{-1})\rVert  & \text{if } t < s, \\
    \frac{s}t \lVert (I - s^2\Delta_w)^{-1}Q_t^w \rVert + \lVert E_t^w Q_s^w \rVert  & \text{if } s < t.
  \end{cases}
\end{equation}
Apply the trace lemma \cite[Lemma 3.13]{ARR} to the last term to conclude.

The novelties are in the vector part, because the actions of $P_t$ and $E_t$ on vectors are different.
Since $u \in \overline{\ran(\homg{D})}$, write $u_{\mathcal{C}} = [u_{\mathbb{R}}, u_N]^\transpose = \nabla g$ for some $g$, and let $h \coloneqq (-\Delta)^{1/2}g$, so that $\nabla g = \mathcal{R}h$.
By \eqref{eq:interaction_P_t_and_Riesz} and holomorphic functional calculus as in the scalar case above, we have
\begin{equation*}
  \lVert E_t ((I - t^2\Delta_H)^{-1} - I) Q_s^H \mathcal{R} \rVert \lesssim
  \begin{cases}
    \frac{t}s \lVert Q_t \, s^2(-\Delta)(I - s^2\Delta)^{-1}\rVert  & \text{if } t < s, \\
    \frac{s}t \lVert \mathcal{R} Q_t (I - s^2\Delta)^{-1} \rVert + \lVert E_t \mathcal{R} Q_s \rVert  & \text{if } s < t.
  \end{cases}
\end{equation*}
Since both $E_t$ and $\mathcal{R}$ are uniformly bounded on $L^2(\mathcal{C},w)$,
the only term left to control is
  \begin{equation*}
    E_t\mathcal{R}Q_s h =
    \begin{bmatrix}
      \fint_Q \nabla_x s(I-s^2\Delta)^{-1} h \D{x}\D{y} \\
      0
    \end{bmatrix}.
  \end{equation*}
  This is controlled via the trace lemma \cite[Lemma 3.13]{ARR}: for some $\theta \in(0,1)$ we have %
  \begin{align*}
    \lVert E_t \mathcal{R} & Q_s h \rVert^2_{L^2(\mathcal{C},w)} = \sum_{Q \in \mathcalboondox{D}_t} w(Q) \left\lvert \fint_{Q} \nabla_x s(I-s^2\Delta)^{-1}h \D{\mu} \right\rvert^2 \\
                                       \lesssim &  \sum_{Q \in \mathcalboondox{D}_t} \frac{w(Q)}{\ell(Q)^{2\theta}} \Big( \fint_{Q} \lvert \nabla_x s(I-s^2\Delta)^{-1}h\rvert^2 w \D{\mu}\Big)^{1-\theta} \Big( \fint_{Q} \lvert s(I-s^2\Delta)^{-1} h \rvert^2 w \D{\mu}\Big)^{\theta}\\
                                       \eqsim & \Big(\frac{s}t\Big)^{2\theta} \sum_{Q \in \mathcalboondox{D}_t} \Big( \int_{Q} \lvert \nabla_x s(I-s^2\Delta)^{-1}h\rvert^2 w \D{\mu}\Big)^{1-\theta} \Big( \int_{Q} \lvert (I-s^2\Delta)^{-1} h \rvert^2 w \D{\mu}\Big)^{\theta}\\
    \le & \Big(\frac{s}t\Big)^{2\theta} \lVert \nabla_x s(I-s^2\Delta)^{-1}h\rVert_{L^2(\mathcal{C},w)}^{2(1-\theta)} \lVert (I-s^2\Delta)^{-1} h \rVert_{L^2(\mathcal{C},w)}^{2\theta} \lesssim  \Big(\frac{s}t\Big)^{2\theta} \lVert h \rVert^2_{L^2(\mathcal{C},w)}
  \end{align*}
  since the operators $\nabla_x s(I-s^2\Delta)^{-1} = \mathcal{R} Q_s$ and $(I-s^2\Delta)^{-1}$ are uniformly bounded on $L^2(\mathcal{C},w)$.  
  This proves the bound
  \begin{equation*}
    \int_1^\infty \left\lVert \gamma_t E_t(P_t - I) \begin{bsmallmatrix}
      u_0 \\ u_{\mathbb{R}} \\ u_N
    \end{bsmallmatrix}\right\rVert_{L^2(\mathcal{C},w)}^2 \frac{\mathrm{d}t}{t} \lesssim \left\lVert \begin{bsmallmatrix}
      u_0 \\ u_{\mathbb{R}} \\ u_N
    \end{bsmallmatrix} \right\rVert_{L^2(\mathcal{C},w)}^2
\end{equation*}
and concludes the proof of the principal part approximation.
\end{proof}

%% file: sections/Carleson_estimate.tex
To conclude the proof of \cref{thm:global_QE_cylinder} we have to show that
\begin{equation*}
  \int_1^\infty \lVert \gamma_t E_t u \rVert_{L^2(\mathcal{C},w)}^2 \frac{\mathrm{d}t}t \lesssim \lVert u \rVert_{L^2(\mathcal{C},w)}^2 .
\end{equation*}
Recall that the principal part $\gamma_t$ maps $ \mathbb{C}^{1+k}$ to $\mathbb{C}^{1+k} \oplus T N$,
since the average operator $E_t$ maps to zero the vertical part of vector fields in $T \mathcal{C}$.
Denote by $\lvert\gamma_t\rvert$ the operator norm on the space of linear maps $\mathscr{L}(\mathbb{C}^{1+k}; \mathbb{C}^{1+k}\oplus T N)$.
Note that this space depends on the point $(x,y)$ where $\gamma_t$ is applied to.
Since $\lvert \gamma_t E_t u \rvert \le \lvert E_t u \rvert \lvert \gamma_t \rvert $,
via the Carleson embedding theorem \cite[Chapter II, Theorem 2]{bigStein}, the bound above reduces to
\begin{equation}\label{eq:Carleson_measure_to_show}
  \int_1^{\ell(Q)} \int_{Q} \lvert \gamma_t(x,y) \rvert^2 w(x,y) \mathrm{d}\mu(x,y) \frac{\mathrm{d}t}t \lesssim w(Q)
\end{equation}
for dyadic cubes $Q \in \mathcalboondox{D}_t$ at scale $t \ge 1$.
The standard way to show such estimate is to perform  
first a sectorial decomposition of the space of maps $\gamma_t$.
On $\mathbb{R}^d$ one exploits the fact that the space of linear maps is a finite dimensional vector space,
and so its unit ball is compact.
The main novelty here is the way we deal with the tangent bundle $T N$. New methods are needed 
since the maps $\gamma_t$ cannot be embedded in a single
finite dimensional vector space.
To overcome this problem, we make use of the 
refined sectorial decomposition from 
\cite[Proposition 4.2]{Andreas_localTb},
which only involves a sectorial decomposition of the
domain space $\mathbb{C}^{1+k}$ of the maps $\gamma_t$.
Since this domain is finite dimensional, the decomposition generalizes to the cylinder.
To our best knowledge, this is the first application of \cite[Proposition 4.2]{Andreas_localTb}
where the standard method of sectorially decomposing the space of matrices is not available.

To show \eqref{eq:Carleson_measure_to_show} we use the local $Tb$ theorem \cite[Theorem 1.1]{Andreas_localTb}, adapted to the cylinder.
We claim that it suffices show the existence of a family of test functions $b_Q^v\in~L^2(\mathcal{C};\mathbb{C}^{1+k},w)$,
one for each dyadic cube $Q \in \mathcalboondox{D}$ and for each unit vector $v\in \mathbb{C}^{1+k}$, such that
\begin{equation}\label{eq:b_Q^v_properties}
  \begin{aligned}
    \Real (v, E_Qb_Q^v) & \gtrsim 1, \\
    \int_{\mathcal{C}} \lvert b_Q^v(x,y)\rvert^2 \mathrm{d}w(x,y) & \lesssim w(Q)\quad\text{and} \\
    \iint_{(1,\infty)\times\mathcal{C}} \lvert \gamma_t(x,y) E_t b_Q^v(x,y) \rvert^2 & \frac{\mathrm{d}t\mathrm{d}w(x,y)}t \lesssim w(Q),
  \end{aligned}
\end{equation}
where we denote by $E_{Q}$ the averaging operator defined in \eqref{eq:averaging_operator_E_t} restricted on $Q$.

To see this, we inspect the proof of \cite[Theorem 1.1]{Andreas_localTb}, which we apply to 
\begin{equation*}
  \widetilde{\gamma_t}(x,y)= \1_{\{t>1\}} \gamma_t(x,y),
\end{equation*}
and the Euclidean cubes $Q \subset \mathbb{R}^n$ are replaced by cubes $Q = P \times N \subset \mathcal{C}$ where $P$ is a cube in $\mathbb{R}^k$ with $\ell(P) \ge 1$. Write $\ell(Q):= \ell(P)$.
The measure $\mu$ in \cite{Andreas_localTb} is replaced by the weighted measure $\mathrm{d}w=w(x,y) \mathrm{d}\mu$, where $\mu$ is the Riemannian measure on the cylinder, see \cref{subsec:space_setup}.

The first step is to decompose $(1,\infty)\times\mathcal{C}$ into finitely many sets
$S(v_0)$, on which $|\gamma_t(x,y)v|\eqsim |\gamma_t(x,y)|$ for all $v$ in a neighbourhood
$D(v_0)\subset\mathbb{C}^{1+k}$ of a fixed unit vector $v_0\in \mathbb{C}^{1+k}$.
The proof of \cite[Proposition 4.2]{Andreas_localTb}
adapts to prove this for the section 
$\widetilde{\gamma_t}(x,y)\in \mathscr{L}(\mathbb{C}^{1+k}; \mathbb{C}^{1+k}\oplus T N_y)$.

The second step, for fixed $v_0\in \mathbb{C}^{1+k}$, with $|v_0|=1$, and cube $Q\subset \mathcal{C}$, with $\ell(Q)\ge 1$, is to perform a 
double stopping time construction
as in \cite[Lemma 2.1]{Andreas_localTb}, with $k=2$, and prove that 
\begin{equation*}
  E_R b^{v_0}_{S_2}\in D(v_0)
\end{equation*}
for all $R\in G^W(S_1)\cap G^b(S_2)$.
Here we view $E_R b^{v_0}_{S_2}\in \mathbb{C}^{1+k}$,
$S_1$ is a stopping cube as \cite[Proposition 4.3]{Andreas_localTb} controlling the averages of the weight, and $S_2\subseteq S_1$ is a stopping cube \cite[Proposition 4.4]{Andreas_localTb} controlling the averages of the test function.
Since $\widetilde{\gamma_t}=0$ for $t<1$, we modify the stopping time constructions to not include any stopping cubes $S$ with $\ell(S)<1$ on $\mathcal{C}$.
The auxiliary matrix weight $W$ in the above stopping constructions is  
  \begin{equation*}
    W = \begin{bmatrix}
      1 & 0 \\
      0 & w^{-1} I_k
    \end{bmatrix}
    : \mathcal{C}\to\mathscr{L}(\mathbb{R}^{1+k}),
  \end{equation*}
and since this has constant principal axes, we note
that instead of
\cite[Proposition 3.6]{Andreas_localTb}
one can use \cite[Proposition 2.1]{ARR}.

The third and last step is to construct test functions
$b^v_Q$ with properties \eqref{eq:b_Q^v_properties} as above.
Note that, even though averages $E_R b^v_Q$ in the stopping time argument above were considered to belong to $\mathbb{C}^{1+k}$ by horizontal projection, the proof of 
\cite[Theorem 1.1]{Andreas_localTb} adapts to allow
for test functions $E_R b^v_Q$ being general sections of the 
bundle $\mathscr{L}(\mathbb{C}^{1+k}; \mathbb{C}^{1+k}\oplus T N)$ over $\mathcal{C}$,
since we assume $v\in \mathbb{C}^{1+k}$.
For $v= (v',0)\in \mathbb{C}^{1+k}$
we define test functions
\begin{equation}\label{eq:test_function}
  b_{Q}^v \coloneqq b_{Q}^{v,\sigma} \coloneqq (I + (\sigma \ell \homg{D} B)^2)^{-1}(\1_{Q} v) \eqqcolon P_{\sigma \ell}^B(\1_{Q} v),
\end{equation}
where $\sigma > 0$ is a parameter to be chosen later.
The following straightforward adaption of 
{\cite[Lemma 3.16]{ARR}}
shows that the average $E_Q(b_Q^v)$ can be made arbitrarily close to $v$.
Note that we still have $E_Q(v)=v$ for $v \in \mathbb{C}^{1+k}$.

\begin{lemma}\label{lemma:3.16ARR}
  There exists  a constant $c > 0$ depending only on
  $\lVert B\rVert_{\infty}$, $w$ and dimension,
  and a constant  $\delta = \delta(w) > 0$
  such that for any cube $Q$ with $\ell(Q)\ge 1$, $v \in \mathbb{C}^{1+k}$ and $\sigma > 0$,
  \begin{equation*}
    \lvert E_{Q} b_{Q}^{v,\sigma} - v \rvert \le c \sigma^\delta .
  \end{equation*}
\end{lemma}  
We are ready to show that $b_Q^v$ satisfy the desired assumptions.
\begin{prop}
  For any unit vectors $v \in \mathbb{C}^{1+k}$
  and each dyadic cube $Q \in \mathcalboondox{D}$
  with $\ell(Q)\ge 1$,
  there exists a test function $b_Q^v \in L^2(\mathcal{C};\mathbb{C}^{1+k},w \mathrm{d}\mu)$ such that
  \begin{enumerate}
  \item $\lVert b_Q^v \rVert_{L^2(\mathcal{C},w)} \lesssim  w(Q)^{1/2}$,
  \item $\Real (v , E_Q b_Q^v) \gtrsim 1$,
  \item the following bound holds
    \begin{equation}\label{eq:property3_test_function}
      \iint_{Q\times(1,\ell(Q))} \lvert \gamma_t(x,y) E_t b_Q^v(x,y) \rvert^2 w(x,y) \mathrm{d}\mu(x,y) \frac{\mathrm{d}t}{t} \lesssim w(Q).
    \end{equation}
  \end{enumerate}
\end{prop}
\begin{proof} Consider the test functions defined in \eqref{eq:test_function}. 
  \begin{enumerate}
  \item The uniform boundedness of $P_t^B$ on $L^2(\mathcal{C},w)$ immediately implies
    \begin{equation}\label{eq:local_L^2_bddness_of_test_function}
      \lVert b_{Q}^v \rVert_{L^2(w)} \lesssim w(Q)^{1/2}.
    \end{equation}

  \item We add and subtract the unit vector $v$,
    then applying Cauchy--Schwarz and \cref{lemma:3.16ARR} to get
    \begin{align*}
      \Real \big(v , E_Q b_Q^v\big) & = \Real \big(v , E_Q(b_Q^v - v)\big) + \lvert v \rvert^2 = \\
                           & = 1 - \Real \big(v , E_Q(v - b_Q^v)\big) \ge 1 - c \sigma^\delta \gtrsim 1.
    \end{align*}
  \item To show \eqref{eq:property3_test_function},
     we use the principal part approximation backwards, following closely the proof of \cite[Theorem 3.3]{ARR}.
     We decompose and bound the integrand
    \begin{align}
       \gamma_t E_t b_{Q}^v & = \gamma_t E_t \big( b_{Q}^v - \1_{Q} v \big) + \gamma_t E_t \1_{Q}v  \nonumber \\
                   & \le \lvert (Q_t^B - \gamma_t E_t) (b_{Q}^v - \1_{Q} v)\rvert \label{eq:terms_for_principal_part1} \\
                   & \quad + \lvert Q_t^B ( b_{Q}^v ) \rvert + \lvert (\gamma_t E_t - Q_t^B)(\1_{Q}v)\rvert,\label{eq:terms_for_principal_part2}
    \end{align}
    where we added and subtracted $Q_t^B\big( b_{Q}^v - \1_{Q} v \big)$ before the inequality.
    For the term in \eqref{eq:terms_for_principal_part1} we can apply \cref{prop:principal_part}
    since $b_{Q}^v - \1_{Q} v \in \overline{\mathsf{im}(\homg{D})}$.
    The first term in \eqref{eq:terms_for_principal_part2}
    equals  $tP_t^B  (\homg{D} B P_{\sigma \ell}^B)(\1_{Q}v)$ with estimate
    \begin{align*}
      \lVert Q_t^B b_{Q}^v \rVert_{L^2(w)} \le \frac{t}{\sigma \ell} \lVert P_t^B \rVert \, \lVert Q_{\sigma \ell}^B \rVert w(Q)^{1/2}\lesssim \frac{t}{\sigma \ell}
      w(Q)^{1/2}.
    \end{align*}
    The last term in \eqref{eq:terms_for_principal_part2} is estimated as in the proof of \cite[Theorem 3.3]{ARR} using off-diagonal estimates for $Q_t^B$ from $\mathcal{C}\setminus Q$ to $Q$,
    and that $\1_Q v \in \mathbb{C}^{1+k}$
    is constant on $Q$.    
  \end{enumerate}
\end{proof}
This concludes the proof of the Carleson estimate \eqref{eq:Carleson_measure_to_show}
and so the proof of \cref{thm:global_QE_cylinder}.

%% file: sections/zero_inj.tex
In this section, we show how \cref{thm:global_QE_cylinder} is used to weaken the geometric hypotheses in~\cite[Theorem 1.1]{AMR}
to the following class of locally cylindrical manifolds.

\begin{define}[Locally Cylindrical Manifolds]\label{def:locally_cylindrical_manifolds}
  Let $(M,g)$ be a complete Riemannian manifold of dimension $d$.
  We make the following hypotheses on the geometry of $M$.
  \begin{enumerate}[itemsep=4mm,leftmargin=0pt,itemindent=*, label=\bfseries (H\arabic*)]
  \item %
    There exists a family $\{\rho_i\}_{i=1}^\infty$ of Lipschitz diffeomorphisms 
    \begin{equation}\label{eq:Lipschitz_diffeo}
      \rho_i \colon (0,1)^{k_i} \times \varepsilon_i N_i \to U_i,
    \end{equation}
    with $k_i \coloneqq d-\text{dim}(N_i)\in\{1,\ldots, d\}$ and $\varepsilon_i\in(0,1)$, between thin cylinders
    $(0,1)^{k_i} \times \varepsilon_i N_i$ with base manifolds $N_i\in\mathcal{N}$,
    and open subsets $U_i\subset M$.
    Here $\mathcal{N}$ is a finite family of closed manifolds of dimensions less than $d$.
    For the Lipschitz constants we assume that
    \begin{equation*}
      \sup_i\max\left\{ \mathrm{Lip}(\rho_i),\mathrm{Lip}(\rho_i^{-1})\right\}=C_1<\infty.
    \end{equation*}
  \item %
    There exists $C_2 \in (0,1)$ such that the sets 
    $K_i= \rho_i([C_2,1-C_2]^{k_i} \times \varepsilon_i N_i)$ form a \emph{uniformly point-finite covering} of $M$, that is,
    $\inf\sum_i \1_{K_i} \ge 1$ and $\sup\sum_i \1_{K_i} <\infty$.
  \item %
    The measure of the $\delta$-fattened set
    $U_i^{\delta}=\{ x\in M ; d(x,U_i)<\delta\}$ grows at most exponentially, that is, there exists $C_3<\infty$
    such that for all $i$ and $\delta\in(0,\infty)$, we have    
    \begin{equation*}
      \mu(U_i^{\delta})\lesssim e^{C_3 \delta} \mu(U_i).
    \end{equation*}
  \end{enumerate}
\end{define}

\begin{remark}
  Note that in (H1) we do not assume any uniform lower bound on the thinness parameters $\varepsilon_i$.
  The collection $\mathcal{N}$ could be allowed to be infinite provided that one assumes suitable uniform estimates of the geometries of these manifolds.

  Hypothesis (H2) is needed %
  when constructing smooth cut-off function for the comparison of resolvents.
  Hypothesis (H3) is needed when using exponential off-diagonal estimates for resolvents to localise the quadratic estimates.
\end{remark}

\begin{lemma}
  Let $(M,g)$ be a locally cylindrical manifold.
  For a fixed $\delta <\infty$, the $\delta$-fattened collection $\{U_i^\delta\}_{i \in \mathbb{N}}$  is also a uniformly point-finite covering of $M$.
\end{lemma}
\begin{proof}
  By hypothesis $1\le \sum_i \1_{K_i}\le M_0$ for some $M_0<\infty$.
  Fix $\delta>0$.
  For a given $x$, let $I(x)$ be the set of indices $i$ for which $x\in U_i^\delta$.
  Since $d(x,U_i)<\delta$ and the diameters of $U_i$ are controlled by $C_1$, it follows that there exists $R<\infty$ independent of $x$, such that $K_i\subset B(x,R)$ for all $i\in I(x)$.
  Thus   
  \begin{equation}\label{eq:bound_the_sum_K_i}
    \sum_{i\in I(x)} \mu(K_i)\le M_0 \mu\big(\bigcup_{i\in I(x)} K_i\big)\le M_0\mu\big(B(x,R)\big).
  \end{equation}
  However, since $B(x,R)\subset U_i^{\delta+R}$, the assumed volume growth shows that $\mu(B(x,R))\le\mu(U_i^{\delta+R})\lesssim \mu(U_i)\lesssim \mu(K_i)$ for $i\in I(x)$, where we used the cylinder geometry up to $C_1$ in the last step.
  This means that there exists a constant $C = C_{R,\delta} < \infty$ such that
  \begin{equation}
    \label{eq:bound_ball_by_K_i}
    \mu\big(B(x,R)\big) \le C \, \mu(K_i)  \quad \text{ for all } i\in I(x).
  \end{equation}
  Then, by using \eqref{eq:bound_ball_by_K_i}
  and \eqref{eq:bound_the_sum_K_i},
  for any $i\in I(x)$ we have that
  \begin{equation*}
    \lvert I(x) \rvert \frac{\mu\big(B(x,R)\big)}{C} \le \lvert I(x) \rvert \mu(K_i) \le \sum_{i \in I(x)} \mu(K_i) \le M_0 \mu\big(B(x,R)\big).
  \end{equation*}
  Thus $|I(x)|\le C M_0$ for some $C<\infty$.
  This proves that
  $1\le \sum_i \1_{K_i}(x)\le\sum_i \1_{U_i^\delta}(x)\le C M_0$ for all $x\in M$,
  and so that $\{U_i^\delta\}_i$ is uniformly point-finite covering too.
\end{proof}

Let $w$ be a Muckenhoupt weight in $A_2^{R_0}(M)$, for some $R_0 >0$, as defined in \cref{def:A_p^R}.
We consider the vector bundle
\begin{equation*}
  \mathcal{V}_M \coloneqq \mathbb{C} \oplus \mathbb{C} \oplus T M
\end{equation*}
and the differential operator $D_M$ acting on $L^2(w)$-sections of $\mathcal{V}_M$ given by
\begin{equation}\label{eq:inhom_D}
  {D}_M \coloneqq
  \begin{bmatrix}
    0 & I & -\frac{1}{w} \mathrm{div} w \\
    I & 0 & 0 \\
    \nabla & 0 & 0
  \end{bmatrix}.
\end{equation}
Let ${B}_M \in L^\infty(\mathsf{End} \mathcal{V}_M)$ be a multiplication operator which is pointwise accretive %
in the sense that
there exists $\kappa_{{B}_M} > 0$ as in \eqref{eq:kappa_in_accretivity} such that
\begin{equation}\label{eq:pointwise_accretivity}
  \Real \big({B}_M(x) v(x) ; v(x) \big) \ge \kappa_{{B}_M} \lvert v(x) \rvert^2 , \qquad \forall v(x) \in (\mathcal{V}_M)_x.
\end{equation}
The inner product and norm above are induced by the Riemannian metric $g$ on $M$.
Our second main result is the following.
\begin{theorem}\label{thm:QE_zero_inj_radius}
  Let $(M,g)$ be a locally cylindrical manifold as in \cref{def:locally_cylindrical_manifolds} satisfying hypotheses $(\mathrm{H}1)$--$(\mathrm{H}3)$.
  Let $w$ be a Muckenhoupt weight in $A_2^{R_0}(M)$, for some $R_0 > 0$.
  Then ${D}_M{B}_M$ is a bisectorial operator
  (in the sense of \cref{prop:typeomega}) %
  and the quadratic estimate
  \begin{equation}\label{eq:QE_M}
    \int_0^{\infty} \big\lVert t {D}_M {B}_M (I + (t {D}_M {B}_M)^2)^{-1} u \big\rVert^2_{L^2(M,w)} \frac{\mathrm{d}t}{t} \lesssim \lVert u \rVert_{L^2(M,w)}^2
  \end{equation}
  holds for all sections $u$ of the bundle $\mathcal{V}_M$ in $L^2(M,w)$.      
\end{theorem}

This result is new, except in the special case when $w= 1$ and the tangent bundle $TN$ has a parallel global frame.
For example, the unweighted case ($w = 1$) in dimension $d = 2$ follows from \cite[Theorem 6.2]{BandaraMcIntoshGBG}
since the tangent bundle $T\mathbb{S}^1$ has a parallel global frame.
However, even for $d=3$ and $N=\mathbb{S}^2$, there are no global frames on 
$T\mathbb{S}^2$, and the GBG charts required in \cite[Theorem 6.2]{BandaraMcIntoshGBG}
do not exist.
More generally, for \cite[Theorem 6.2]{BandaraMcIntoshGBG} to apply in our geometric setting,
a global frame for $TN$ which is parallel must exist. 
Else the gradient bounds in item (iv) in \cite[Theorem 6.2]{BandaraMcIntoshGBG} will fail on thin cylinders 
$(0,1)^k \times \varepsilon N$ as $\varepsilon \to 0$, and the Poincar\'e inequality for vector fields in \cite[Proposition 5.3]{BandaraMcIntoshGBG} will not hold 
uniformly over $M$ since the constant $C_G\to \infty$ as
$\varepsilon \to 0$.

\begin{remark}
  Note that \cref{thm:QE_zero_inj_radius} strictly contains \cite[Theorem 1.1]{AMR},
  which assumes that $M$ has Ricci curvature and injectivity radius bounded from below.
  Assumptions (H1)--(H3) imply exponential volume growth, %
  see references in \cite[\S 2.1]{AMR},
  and geodesic balls on $M$ are bilipschitz equivalent to euclidean balls by the Anderson--Cheeger result \cite[Theorem 0.3]{AndersonCheeger92}
  as in \cite[Equation (2.2)]{AMR}.
  Hence, we may simply set $N = \{0\}$ and $k = d$ for the covering of $M$.
\end{remark}

\begin{remark}\label{rmk:Gårding-inequality}
  Instead of pointwise accretivity for ${B}_M$ as in \eqref{eq:pointwise_accretivity},
  it suffices to assume a Gårding inequality
  as in \cite[Equation (2.7)]{AMR}. 
  This weaker assumption is possible at the price of constructing a more complicated extension of the pulled-back coefficient $B_\rho$, see \cite[Equation (3.8)]{AMR}.
  We refer to \cite[Lemma 3.5 and \S 3.3]{AMR},
  where accretivity of the extended coefficients is shown,
  and we leave the extension and adaptation of~\cite[Lemma 3.5]{AMR} to our geometric setting to the interested reader.
\end{remark}

\subsection{Proof of \cref{thm:QE_zero_inj_radius}}\label{subsec:proof_zero_inj}
The fact that the operator ${D}_M{B}_M$ is bisectorial follows by self-adjointness of ${D}_M$
and accretivity of $B_M$.
The former property follows by an adaption of \cite[Lemma 2.3]{AMR} to our geometric setting.
  
The rest of this section is devoted to the proof of the quadratic estimate \eqref{eq:QE_M},
which is divided in three parts:
\begin{enumerate}[label=(\alph*)]
\item The global quadratic estimate \eqref{eq:QE_M} is reduced to a \emph{local} quadratic estimate; %
\item the local estimate is pulled-back and extended to a %
  cylinder $\mathcal{C}_{\varepsilon} \coloneqq \mathbb{R}^k \times \varepsilon N$;
\item the local quadratic estimate on $\mathcal{C}_{\varepsilon}$ is derived from~\cref{thm:global_QE_cylinder}.
\end{enumerate}

Here $\varepsilon N$ denotes the Riemannian manifold $(N, \varepsilon^2 g_N)$.

\subsubsection{(a): From global to local}
To show estimate \eqref{eq:QE_M},
it is enough to prove the following \emph{local} quadratic estimate 
\begin{equation}\label{eq:local_QE_M}
  \int_0^{1} \big\lVert t {D}_M {B}_M (I + (t {D}_M {B}_M)^2)^{-1} ( u \1_{K^{\alpha}}) \big\rVert^2_{L^2(K,w)} \frac{\mathrm{d}t}{t} \lesssim \lVert u \rVert_{L^2(K^{\alpha},w)}^2 \,,
\end{equation}
for the $\alpha$-fattening $K^\alpha$ of a compact set $K = K_i$ in the covering of $M$.
This estimate controls the action of the operator on small scales and localised functions. %
This reduction is possible in view of the following known arguments:
\begin{enumerate}
\item Quadratic estimate \eqref{eq:QE_M}
  for inhomogeneous operator ${D}_M {B}_M$ holds for large scales ($t\ge 1$) by \cite[Lemma 2.4]{AMR}.
\item Localisation to sets $K \subseteq M$ in the covering
  follows by applying off-diagonal estimates for $D_MB_M$
to off-set the exponential volume growth.
See \cite[\S 3.1]{AMR} for the argument, which is straightforward to adapt to our geometric setting.
\end{enumerate}

We follow \cite[\S 3]{AMR}.
The $\alpha$-fattenings $\{ K_i^{\alpha} \}_i$ form a finite overlapping family for any finite $\alpha >0$.
Since $K_i \subset U_i$, define $\alpha_1 \coloneqq \sup \{ \alpha > 0 \,\colon\, K_i^{\alpha} \subset U_i\}$, so that $K_i^{\alpha_1} \subset U_i$ for each $i$.
Let $\psi_t(D_M B_M) \coloneqq t {D}_M {B}_M (I + (t {D}_M {B}_M)^2)^{-1}$.
We fix $\alpha \in (0,\alpha_1)$ and decompose the manifold using the covering $\{K_i\}_i$ and
then split the function $u$ on $K_i^{\alpha}$ and its complement: 
\begin{align*}
  \int_0^1 & \|\psi_t({D}_M {B}_M)u\|_{L^2(M, w)}^2 \frac{\mathrm{d}t}{t} \\
  & \lesssim \int_0^1 \sum_{i\in\mathbb{N}} \left(\|\1_{K_i} \psi_t({D}_M {B}_M) \1_{K_i^{\alpha}} u\|_{L^2(M, w)}^2 + \|\1_{K_i} \psi_t({D}_M {B}_M) \1_{(K_i^{\alpha})^\complement} u\|_{L^2(M, w)}^2\right) \frac{\mathrm{d}t}{t}\\
  & \leq \sum_{i\in\mathbb{N}} \int_0^1 \| \psi_t({D}_M {B}_M) \1_{K_i^{\alpha}} u\|_{L^2(K_i, w)}^2 \frac{\mathrm{d}t}{t} + \int_0^1 \sum_{i\in\mathbb{N}} \lVert \psi_t({D}_M {B}_M) \1_{(K_i^{\alpha})^\complement} u\|_{L^2(K_i, w)}^2 \frac{\mathrm{d}t}{t} \\
  & \eqqcolon \romannum{1}_{\mathrm{loc}} + \romannum{1}_{\mathrm{off}}.
\end{align*}
We proceed as in the proof for quadratic estimates on manifolds in \cite[Theorem 1.1]{AMR}.
For $\romannum{1}_{\mathrm{off}}$ we use exponential off-diagonal estimates, which hold in our geometric setting, as in \cite[\S 3.1]{AMR} to 
off-set the exponential volume growth (H3).

\subsubsection{(b): Pull-back and extensions to $\mathcal{C}_{\varepsilon}$}
The local quadratic estimate \eqref{eq:local_QE_M},
which corresponds to the local term $\romannum{1}_{\mathrm{loc}}$,
is pulled-back as in \cite[\S 3.2]{AMR} to a cylinder $\mathcal{C}_{\varepsilon} \coloneqq \mathbb{R}^k \times \varepsilon N$, for some $\varepsilon \in (0,1)$.
Note that the pull-back in \cite[\S 3.2]{AMR} is to $\mathbb{R}^d$, but this can be replaced by any manifold,
in particular $\mathcal{C}_{\varepsilon}$.

Let $\rho_i$ the Lipschitz diffeomorphism given in \eqref{eq:Lipschitz_diffeo}.
In the following, we will suppress the index $i$ in $\rho_i$ and its inverse when it is implicit in the context.
Let $D_\rho$ and $B_\rho$ be the pull-back of $D$ and $B$ defined on $(0,1)^{k} \times \varepsilon N$
via $\rho$ as defined in \cite[\S 3.2]{AMR}.
Let $w_\rho$ be the pull-back $w_\rho(x) \coloneqq (w \circ \rho)(x)$ of $w$ via $\rho$.

  We extend the pulled-back weight $w_\rho \colon (0,1)^k \times N \to \mathbb{R}$ to the whole cylinder $\mathbb{R}^k \times N$,
  by considering the even, periodic extension of $x \mapsto w_\rho(x,y)$ for fixed $y \in N$.
  This extension is a Muckenhoupt weight in $A_2$, by a straightforward extension of \cref{prop:periodic_even_ext} to weights on $\mathbb{R}^k\times N$.

We extend the pulled-back operator to the whole cylinder $\mathcal{C}_{\varepsilon}$.
Let $I$ be the identity operator.
With a small abuse of notation,
we define the extended operators $D_\rho$ and $B_\rho$ to the whole cylinder as
\begin{align*}
  D_\rho & \coloneqq D_{\mathcal{C}_{\varepsilon}} =
       \begin{bmatrix}
         0 & I & - (1/w_{\rho}) \mathrm{div} w_{\rho} \\
         I & 0 & 0 \\
         \nabla & 0 & 0
       \end{bmatrix} , \\
  B_\rho & \coloneqq B_\rho \1_{(0,1)^k \times \varepsilon N} + I \1_{(\mathbb{R}^k \setminus (0,1)^k) \times \varepsilon N}.
\end{align*}
This extension preserves boundedness and pointwise accretivity.

\subsubsection{(c): Local quadratic estimate on $\mathcal{C}_{\varepsilon}$}
Finally, we show how the local quadratic estimate on the thin cylinder $\mathcal{C}_{\varepsilon}$ is deduced from the estimate of $D^0B$ on $\mathcal{C}$.
  
  \begin{prop}[Local quadratic estimates on $\mathcal{C}_{\varepsilon}$]\label{lemma:local_QE}
    Let $N$ be a closed manifold and let $\mathcal{C}_{\varepsilon} \coloneqq \mathbb{R}^k \times \varepsilon N$, for $\varepsilon \in (0,1)$ and $k \in \mathbb{N}$.
    Let $w \in A_2(\mathcal{C}_\varepsilon)$ and $B \in L^\infty(\mathsf{End}(\mathcal{V}_{\mathcal{C}_\varepsilon}))$ be pointwise accretive as in \eqref{eq:pointwise_accretivity}.
    Then the weighted local quadratic estimate 
    \begin{equation*}
      \int_0^{1} \big\lVert t D B (I + (t D B)^2)^{-1} u \big\rVert^2_{L^2(\mathcal{C}_\varepsilon,w)} \frac{\mathrm{d}t}{t} \lesssim \lVert u \rVert_{L^2(\mathcal{C}_\varepsilon,w)}^2
    \end{equation*}      
    holds,
    where 
    \begin{equation*}
      D \coloneqq
      \begin{bmatrix}
        0 & I & -\frac{1}{w} \mathrm{div} w \\
        I & 0 & 0 \\
        \nabla & 0 & 0
      \end{bmatrix}.
    \end{equation*}
      
  \end{prop}
  
  \begin{proof}[Proof of \cref{lemma:local_QE}]
    For a fixed $\varepsilon \in (0,1)$ local quadratic estimates for $D B$ on the cylinder $\mathcal{C}_{\varepsilon}$ holds by \cite[Theorem 1.1]{AMR},
    since the cylinder $\mathcal{C}_{\varepsilon}$ has Ricci curvature bounded from below and positive injectivity radius.
    The problem here is to handle the scaling parameter $\varepsilon$ uniformly,
    since as $\varepsilon \to 0$, so does the injectivity radius $\mathrm{inj}(\mathcal{C}_{\varepsilon})$.
    The key idea here is that the local quadratic estimate for $D B$ can be \emph{rescaled}.
    For $\lambda > 0$ consider the dilation 
    \begin{equation}
      \label{eq:dilation}
      \begin{aligned}
        \mathsf{dil}_\lambda &\colon \mathbb{R}^k\times N \to  \mathbb{R}^k\times \lambda N, \\
                      & (x,y) \mapsto (\lambda x, \lambda y).
      \end{aligned}
    \end{equation}
    Recall that $\lambda N$ denotes the Riemannian manifold
    $(N,\lambda^2 g_N)$,
    and we write $\lambda y \coloneqq y\in \lambda N$.
    In particular, $\mathsf{dil}_{1/\varepsilon}(\mathcal{C}_\varepsilon) = \mathcal{C}$.
    For $u \in W^{1,2}_{\mathrm{loc}}(\mathcal{C}_{\varepsilon})$, we denote the rescaled function on $\mathcal{C}$ by $u_\varepsilon(x,y) \coloneqq u(\varepsilon x,\varepsilon y)$, 
    then we have
    \begin{equation*}
      (\nabla u)(\varepsilon x,\varepsilon y) = \varepsilon^{-1} (\nabla u_{\varepsilon})(x, y).
    \end{equation*}  
    Let $w \in A_2(\mathcal{C}_\varepsilon)$. Since the $A_2$ condition is invariant under dilation,
    the rescaled weight $w_\varepsilon$ is in $A_2(\mathcal{C})$.
    If we consider smooth sections $u$ of the bundle
  \begin{equation*}
    \mathcal{V}_{\mathcal{C}_{\varepsilon}} \coloneqq \mathbb{C} \oplus \mathbb{C} \oplus T \mathcal{C}_{{\varepsilon}},
  \end{equation*}
  the inhomogeneous operator $D$ on $\mathcal{C}_{\varepsilon}$ defined in \eqref{eq:inhom_D} is similar to
  \begin{equation*}
    \tfrac{1}{\varepsilon} D^\varepsilon \coloneqq
    \begin{bmatrix}
      0 & I & -\frac{1}{\varepsilon}\, \frac{1}{w_\varepsilon}\mathrm{div} w_\varepsilon \\
      I & 0 & 0 \\
      \frac{1}{\varepsilon}\nabla & 0 & 0
    \end{bmatrix}.
  \end{equation*}
  On $\mathcal{C}$ under the weighted $L^2$ isometry  $\mathsf{dil}_\varepsilon^*: u(x,y)\mapsto \varepsilon^{-d/2} u(\varepsilon x, \varepsilon y)$.
  Write $B_\varepsilon \coloneqq \mathsf{dil}_\varepsilon^* B (\mathsf{dil}_\varepsilon^*)^{-1}$ 
  for the multiplier on $\mathcal{C}$ corresponding to $B$.

  By similarity and the change of variables $s = t/\varepsilon$ it suffices to bound
  \begin{equation*}
    \int_0^{1} \big\lVert \psi(\tfrac{t}{\varepsilon} D^{\varepsilon} B_{\varepsilon}) v \big\rVert^2_{L^2(\mathcal{C},w_{\varepsilon})} \frac{\mathrm{d}t}{t} =
    \int_0^{1/\varepsilon} \big\lVert \psi(s D^{\varepsilon} B_{\varepsilon}) v \big\rVert^2_{L^2(\mathcal{C},w_\varepsilon)} \frac{\mathrm{d}s}{s}
  \end{equation*}
  in terms of $\lVert v\rVert^2_{L^2(\mathcal{C},w_\varepsilon)}$ for
  a section $v$ of the bundle $\mathcal{V}_{\mathcal{C}}$.
  Let $\homg{D}$ be the homogeneous operator on the cylinder $\mathcal{C}$ introduced in \cref{subsec:preamble_main_thm}.
  By a small abuse of notation, we denote by $\homg{D}$ also its augmented version
  \begin{equation*}
    \homg{D} \coloneqq
    \begin{bmatrix}
      0 & 0 & -\frac{1}{w_\varepsilon}\mathrm{div} w_\varepsilon \\
      0 & 0 & 0 \\
      \nabla & 0 & 0
    \end{bmatrix}
  \end{equation*}
  to match the size of the coefficients $B_{\varepsilon}$. %
  Quadratic estimates for $D^0 B_{\varepsilon}$ holds by \cref{thm:global_QE_cylinder}
  for any bounded, pointwise accretive $B_{\varepsilon}\in L^\infty(\mathsf{End}(\mathcal{V}_{\mathcal{C}}))$ and any weight $w \in A_2(\mathcal{C})$.
  To see this for general $B_\varepsilon$, note that
  it suffices, by (ii) in \cref{prop:typeomega}, to show these quadratic estimate on 
  $\ran(D^0B_\varepsilon)\subset \mathcal{H}_1$, where
  $\mathcal{H}_1$ are the $L^2(w)$ sections of
  $\mathcal{V}_{\mathcal{C}}$ with zero second component. And on $\mathcal{H}_1$, the operator
  $D^0B_\varepsilon$ coincides with an operator
  that  \cref{thm:global_QE_cylinder} applies to.

  It remains to compare the resolvents of the two operators.
  Let $R_t(T) \coloneqq (I + i t T)^{-1}$ be the resolvent of $T$.
  The difference of the corresponding resolvents in the (rescaled) local quadratic estimate is integrable:
  \begin{align*}
    \int_0^{1/\varepsilon} & \Big\lVert \Big( R_s\big(D^{\varepsilon} B_{\varepsilon}\big) - R_s\big(\homg{D} B_{\varepsilon}\big) \Big)v \Big\rVert^2 \frac{\mathrm{d}s}{s} \\
            & \le \int_0^{1/\varepsilon} s^2 \lVert D^{\varepsilon} - \homg{D} \rVert^2 \, \lVert B_\varepsilon\rVert^2 \lVert v \rVert^2 \frac{\mathrm{d}s}{s}
              \lesssim \varepsilon^2 \int_0^{1/\varepsilon} s \lVert v \rVert^2 \,\mathrm{d}s \lesssim \lVert v \rVert^2
  \end{align*}
  since the difference between the homogeneous and inhomogeneous Dirac operator is the bounded operator 
  \begin{equation*}
    D^{\varepsilon} - D^0 =
    \begin{bmatrix}
      0 & \varepsilon I & 0 \\
      \varepsilon I & 0 & 0 \\
      0 & 0 & 0
    \end{bmatrix}.
  \end{equation*}
  Note also that the rescaled multiplication operator $B_{\varepsilon}(x)$ is uniformly bounded and accretive for $\varepsilon \in (0,1]$.
  Since $\psi(s T)=(R_{-s}(T)-R_s(T))/(2i)$, 
  this concludes the proof of the local quadratic estimates in \cref{lemma:local_QE}.
\end{proof}

\begin{proof}[{Proof of \cref{thm:QE_zero_inj_radius}}]
  The quadratic estimate in \cref{lemma:local_QE} implies the quadratic estimate \eqref{eq:local_QE_M}
  by adapting the argument in  
  \cite[Proof of Theorem 1.1]{AMR} to our geometric setting.
  The idea is to estimate the difference of the two
  $DB$ resolvents locally on the chart
  $\rho_i \colon (0,1)^{k_i} \times \varepsilon_i N_i \to U_i$ from (H1).
  See also \cite[Lemma 7.2]{AKMc06}.
  This concludes the proof of the theorem.
\end{proof}

%% file: sections/poincare-estimates.tex
We show weighted Poincaré inequalities on the cylinder $\mathcal{C} = \mathbb{R}^k \times N$,
where $N$ is a closed, connected, $n$-dimensional Riemannian manifold.
We denote by $\mathrm{d}w$ the weighted measure $w(x,y) \D{x}\D{y}$.

The strategy is to decompose the cylindrical manifold into slender rectangles (slabs),
then prove a weighted Poincaré inequality on each slab and derive a global inequality by patching the slabs together.

\subsection{Poincaré for a slab}
We prove the Poincaré inequality on a single slab of the cylinder.
Let $\{N_i, i = 1,\dots, m\}$ be a finite open covering of the compact base $N$ of the cylinder $\mathcal{C}$.
  A dyadic cube $Q \in \mathcalboondox{D}_t(\mathcal{C})$, for $t > \mathrm{diam}(N) \eqsim 1$ is $Q=R \times N$
  where $R$ is a dyadic cube in $\mathbb{R}^k$.
  Assume that the horizontal component is $R = [0,h]^k$, 
  and so $\ell(Q) = h$.
  We extend the sets $N_i$ in the covering in the horizontal direction,
  considering $Q_i \coloneqq  [0,h]^k \times N_i$.
 
  We rewrite the average of $u$ on $Q$ as
  \begin{equation}\label{eq:average_via_partition_unity}
    \begin{aligned}
    \langle u \rangle_{Q} & = \frac{1}{h^k\lvert N\rvert} \iint_{[0,h]^k\times N}  u(x,y) \D{x}\mathrm{d}y \\
    & \le \sum_{i=1}^m \frac{\lvert [0,h]^k\times N_i \rvert}{h^k \lvert N\rvert} \Big(\frac{1}{\lvert Q_i\rvert} \iint_{Q_i} u \Big) \le \sum_{i=1}^m \langle u \rangle_{Q_i},
    \end{aligned}
  \end{equation}
  since $\lvert Q_i\rvert/\lvert Q \rvert \le 1$.
  Note that, since the family $\{N_i\}_i$ is finite, in particular is finitely overlapping.
  
  \begin{lemma}[Poincaré on a single slab]
    Let $h \ge 1$ and $w \in A_2(\mathcal{C})$.
    Then for all $u \in H^1(\mathcal{C},w)$ we have
    \begin{equation}\label{eq:poincare_local_Ni}
      \fiint_{[0,h]^k \times N_i} \lvert u - \langle u \rangle_{Q_i} \rvert^2 \D{w} \lesssim h^2 \fiint_{[0,h]^k \times N_i} \lvert \nabla_{x,y} u\rvert^2 \D{w},
    \end{equation}
    where $\langle u \rangle_{Q_i}$ is the unweighted average of $u$ on $[0,h]^k\times N_i$.
  \end{lemma}

  \begin{proof}[Proof of~\eqref{eq:poincare_local_Ni}] %
    The idea is to pull back the inequality on the slab to a parallelepiped in $\mathbb{R}^d$.
    By considering even, periodic extensions of the pulled-back function and weight,
    we extend the inequality to a cube.
    The weighted Poincaré inequality holds on cube,
    and it is nothing but a multiple of the inequality on the parallelepiped, which holds as well.
    We give details below.

    We will pull back to cubes in $\mathbb{R}^n$ via diffeomorphisms
    \begin{equation}
      \label{eq:pullback_to_Euclidean_cubes}
      \varphi_i \colon [0,1]^n \to N_i.
    \end{equation}    
    After a change of variables with the diffeomorphism $\varphi_i$ in \eqref{eq:pullback_to_Euclidean_cubes}
    and renaming the pullback via $\varphi_i$, %
    equation \eqref{eq:poincare_local_Ni} is equivalent to
    \begin{equation}\label{eq:poincare_eucli_rect}
      \iint_{[0,h]^k \times [0,1]^n}  \lvert g - \langle u \rangle_{Q_i} \rvert^2 w_i(x,\upsilon) \mathrm{d}x\mathrm{d}\upsilon \lesssim h^2 \iint_{[0,h]^k \times [0,1]^n} \lvert \nabla_{x,y} g\rvert^2 w_i(x,\upsilon) \mathrm{d}x\mathrm{d}\upsilon
    \end{equation}
    where $g(x,\upsilon) \coloneqq u(x,\varphi_i(\upsilon))$, $w_i(x,\upsilon) \coloneqq w(x,\varphi_i(\upsilon))$.
    The averages $\langle u \rangle_{Q_i}$ can be written in terms of averages of $g$ on $[0,h]^k \times [0,1]^n$ with respect to the weighted measure $J_{\varphi_i} \mathrm{d}x\mathrm{d}\upsilon$ since
    \begin{equation*}
     \iint_{[0,h]^k \times N_i} u(x,y) \mathrm{d}x\mathrm{d}y = \iint_{[0,h]^k \times [0,1]^n} g(x,\upsilon) J_{\varphi_i} \mathrm{d}x\mathrm{d}\upsilon ,
   \end{equation*}
   and we have that $\lvert [0,h]^k \times N_i \rvert = J_{\varphi_i}\big( [0,h]^k \times [0,1]^n \big)$. So we have
   \begin{equation*}
     \langle u \rangle_{Q_i} = \frac{1}{J_{\varphi_i}\big( [0,h]^k \times [0,1]^n \big)} \iint_{[0,h]^k \times [0,1]^n} g(x,\upsilon) J_{\varphi_i} \mathrm{d}x\mathrm{d}\upsilon = \langle g \rangle_{[0,h]^k \times [0,1]^n}
   \end{equation*}
   where the last average is taken with respect to the weighted measure $J_{\varphi_i} \mathrm{d}x\mathrm{d}\upsilon$.
    The Jacobian $J_{\varphi_i}$ is uniformly bounded above and below,
    and the proof of~\cite[Lemma 1.4 and Theorem 1.5]{FabesKenigSerapioni82} work as well for~\eqref{eq:poincare_eucli_rect} when $h = 1$.  
    
    Finally, by adjusting the diffeomorphism $\varphi_i$, without loss of generality
    we can assume $h$ to be an integer.
    The weighted Poincaré inequality holds on the cube $[0,h]^{d}$ for any $A_2$ weight by~\cite[Theorem 1.5]{FabesKenigSerapioni82}.
    The cube $[0,h]^{d}$ is made of $h^n$ copies of the parallelepiped $[0,h]^k \times [0,1]^n$.
    By \cref{prop:periodic_even_ext}, the even, periodic extension in $\mathbb{R}^n$ of $w_i$ is in $A_2$. 
    This implies the weighted Poincaré inequality for the even, periodic extension of $g$ and $w_i$ on the cube,
    and so~\eqref{eq:poincare_eucli_rect}.
\end{proof}

\subsection{Merging slabs} %
\label{subsec:poincare}
We prove  weighted Poincaré inequality on the cylinder $\mathcal{C}$ in \cref{lemma:poincare_scalar}  %
by merging the inequalities obtained on each slab.

\begin{proof}[Proof of \cref{lemma:poincare_scalar}]
  Assume that the horizontal component is $R = [0,h]^k$, and so $Q = R \times N$, with $\ell(Q) \eqsim h$.
  Bound
  \begin{equation*}
    \int_{N} \int_{[0,h]^k} \lvert u(x,y)-\langle u\rangle_{Q}\rvert^2 \D{w} \leq \sum_{i=1}^m \int_{N_i} \int_{[0,h]^k} \lvert u(x,y)-\langle u\rangle_{Q}\rvert^2 \D{w}
  \end{equation*}
  Denote by $\langle u\rangle_{i}$ the average $\langle u\rangle_{Q_i}$ where $Q_i = [0,h]^k \times N_i$.
  The integrand in  each term is controlled by
  \begin{equation*}      
    \lvert u(x,y)-\langle u\rangle_{Q}\rvert^2 \le 2 \big( \lvert u(x,y)- \langle u\rangle_{i}\rvert^2 + \lvert \langle u\rangle_{Q} -\langle u\rangle_{i}\rvert^2 \big).
  \end{equation*}
  The first term is estimated by~\eqref{eq:poincare_local_Ni}.
  For the latter, we control the difference of averages by writing $\langle u \rangle_{Q}$ as in~\eqref{eq:average_via_partition_unity}
  \begin{align*}
    \lvert \langle u\rangle_{Q}-\langle u\rangle_i \rvert^2 &\le \Big( \sum_{j=1}^m \fiint_{Q_j} \lvert u-\langle u\rangle_i \rvert \Big)^2 \lesssim \sum_{j=1}^m \fiint_{Q_j} \lvert u-\langle u\rangle_i \rvert^2 \\
    & \lesssim \sum_{j=1}^m \fiint_{Q_j} \lvert u-\langle u\rangle_j \rvert^2 + \sum_{j=1}^m \lvert \langle u\rangle_j -\langle u\rangle_i \rvert^2.
  \end{align*}
  Again the first term is estimated by~\eqref{eq:poincare_local_Ni}.
  For the latter we distinguish two cases: when $N_j$ intersects $N_i$, and when not.
    If $N_i \cap N_j \neq \emptyset$, since $\{N_j\}_j$ is a finite open covering of $N$, the intersection is an open set.
    By the doubling property,
    $w([0,h]^k \times N_i) \lesssim w([0,h]^k \times (N_i \cap N_j))$ for any $i,j$ for which $N_i \cap N_j$ is not empty.
    By adding and subtracting $u$, we get
    \begin{align*}
      \lvert \langle u\rangle_j-\langle u\rangle_i\rvert^2 & w([0,h]^k \times N_i) \lesssim \iint_{[0,h]^k \times (N_i\cap N_j)} \lvert u - \langle u\rangle_j\rvert^2 + \lvert u -\langle u\rangle_i\rvert^2  \D{w} \\
                                                                           & \le  \iint_{[0,h]^k \times N_j} \lvert u - \langle u\rangle_j\rvert^2 \D{w} + \iint_{[0,h]^k\times N_i} \lvert u -\langle u\rangle_i\rvert^2 \D{w} \\
                            & \lesssim h^2 \lVert \1_{N_i} \nabla u \rVert_{L^2([0,h]^k\times N, w)}^2 + h^2 \lVert \1_{N_j} \nabla u \rVert_{L^2([0,h]^k\times N, w)}^2 \\
                            & \le 2 h^2 \, \lVert \nabla u \rVert_{L^2([0,h]^k\times N, w)}^2 
    \end{align*}
    where we used~\eqref{eq:poincare_local_Ni} twice.
    If $N_i \cap N_j = \emptyset$, then we argue as in the previous case by adding a
    telescoping sum on a finite number of averages on pairwise intersecting slabs $[0,h]^k\times N_j$.
    This is possible since we assume $N$ to be connected.
 
    This completes the proof of the weighted Poincaré inequality with unweighted averages.
    As noted in \cite{FabesKenigSerapioni82}, one can replace the unweighted averages $\langle u \rangle_i$ with weighted averages:
    the difference between the two averages is bounded by the left-hand side in the Poincaré inequality \eqref{eq:poincare_local_Ni}.
    This completes the proof.
  \end{proof}

%% file: sections/singular_integrals.tex
Here we gather proofs of the weighted estimates for singular integral operators on cylinders $\mathbb{R}^k\times N$,
where $N$ is a closed, connected Riemannian manifold.
The cylinder $\mathcal{C} = \mathbb{R}^k \times N$ with the product metric $\mathrm{d}(\cdot,\cdot)$ and measure $\mu$ introduced in \cref{subsec:space_setup} is a space of homogeneous type (SHT).
We show that the Riesz transform $\nabla (-\Delta)^{-1/2}$ and the Beurling transform $\overline{\nabla} \nabla (-\Delta)^{-1}$ on the cylinder $\mathcal{C}$
are Calderón--Zygmund kernels operators on a space of homogeneous type.
Weighted estimates for these operators, even quantitative such, which we shall not need, follow via sparse domination,
for example by applying~\cite[Theorem 6.1]{Lorist}.

  Let $(\mathcal{C},\mathrm{d},\mu)$ be a space of homogeneous type as defined in \cref{subsec:space_setup}.
  Let $B(x,r)$ be the ball of center $x$ and radius $r$ in $\mathcal{C}$.
  Note that for all $x_0 \in \mathcal{C}$ we have
  \begin{equation*}
    \mu\left(B(x_0, r)\right) \eqsim \min(r^k , r^d).
  \end{equation*}
  
  \begin{define}[Calderón--Zygmund operator on $\mathcal{C}$]\label{def:CZ}
    Let $X, Y$ be two Banach spaces.
    We say that $T$ is a Calderón--Zygmund operator on the cylinder $(\mathcal{C},\mathrm{d},\mu)$
    if $T$ is a linear, bounded operator
    $T \colon L^2(\mathcal{C};X) \to L^2(\mathcal{C}, Y)$
    and there exists a map $k \colon \mathcal{C} \times \mathcal{C} \setminus \{(x,x) \,\colon\, x\in\mathcal{C} \} \to \mathscr{L}(X,Y)$ such that
    \begin{equation*}
      T f(x) = \int_{\mathcal{C}} k(x,y) f(y) \mathrm{d}\mu(y)
    \end{equation*}
    for all $x \not\in \supp f$, and
    the map $k$ is a Calderón--Zygmund kernel,
    meaning that there exists a finite constant $c_T > 0$ such that $k$ satisfies the size and regularity estimates
    \begin{align}
      \lVert k(x,y) \rVert & \lesssim c_T \min(r^k , r^d)^{-1}, \qquad \text{and}\label{eq:kernel_estimate_CZ_SHT} \\
      \lVert \nabla_x k(x,y) \lVert + \lVert \nabla_y k(x,y) \lVert  & \lesssim c_T r^{-1} \min(r^k , r^d)^{-1} \label{eq:regularity_kernel_estimate_CZ_SHT}.
    \end{align}
    where $r \coloneqq d(x,y)$.
    The implicit constants depend only on $\mathcal{C}, \mathrm{d}$, and the measure $\mu$.
\end{define}

Thus, to prove weighted estimates for an operator $T$ it is enough to check:
\begin{enumerate}
\item the unweighted $L^2$-boundedness of $T$; %
\item the size estimate \eqref{eq:kernel_estimate_CZ_SHT}; %
\item the regularity estimate \eqref{eq:regularity_kernel_estimate_CZ_SHT}. %
\end{enumerate}

\begin{remark}
  The conditions \eqref{eq:kernel_estimate_CZ_SHT} and \eqref{eq:regularity_kernel_estimate_CZ_SHT} are technically stronger than the one in \cite[Theorem 6.1]{Lorist},
  where the control of $\nabla k$ is relaxed to a Dini condition.
  Our conditions above imply the Dini condition in \cite[\S 6]{Lorist} for Dini kernel with $\omega(t) = t$.  
\end{remark}
The Calderón--Zygmund kernels for our singular integral operators are obtained using representation formulas with the heat kernel,
and estimates \eqref{eq:kernel_estimate_CZ_SHT} and \eqref{eq:regularity_kernel_estimate_CZ_SHT} are deduced from estimates for the heat kernel on $\mathcal{C}$.

\subsection{Heat kernel estimates}\label{subsec:heat_kernel_estimates}

\input{sections/heat_kernel}

\subsection{Riesz transform bounds}
\label{subsec:Riesz_transform_bounds}

The Riesz transform bounds on cylinders follow by combining the heat kernel estimates from
\cref{lemma:pointwise_bounds_heat_kernel}, 
the unweighted estimates from~\cite[Theorem 1.4]{auscher2004riesz}
and the weighted estimates from~\cite[Corollary 1.3]{AuscherMartellIV}.
We include a proof below for completeness.

\begin{proof}[Proof of \cref{lemma:Riesz_bounds}] %
  The Riesz transform $\mathcal{R} \coloneqq \nabla (-\Delta)^{1/2}$ is an isometry on unweighted $L^2$, so it suffices to show the size and regularity estimates for the kernel of $\mathcal{R}$. 
  To do so, we write the Riesz transform in terms of the heat propagator $e^{t^2 \Delta}$ as
  \begin{align*}
    \nabla (-\Delta)^{-1/2} & = \frac{2}{\sqrt{\pi}} \int_0^\infty \nabla e^{t^2 \Delta} \D{t}.
  \end{align*}
  To show that this is a Calderón--Zygmund operator we bound the kernel
  \begin{equation*}
    k(x,y) = \int_0^\infty \nabla p(t^2,x,y) \D{t}. %
  \end{equation*}
  Let $r \coloneqq d(x,y)$ be the distance on the cylinder.
  We proceed using the estimates for the heat kernel of $e^{t^2 \Delta}$ and its derivatives as in \cref{lemma:pointwise_bounds_heat_kernel},
  and the change of variables $s = r/t$.
  \begin{align*}
    \int_0^\infty \lvert \nabla p(t^2,x,y)\rvert \mathrm{d}t & \lesssim \int_0^\infty \frac{\exp\big(-\frac{r^2}{ct^2} \big)}{\min(t^d,t^k)} \big( \tfrac{r}{t} + 1 \big) \frac{\mathrm{d}t}{t} \\
                                & = \int_0^\infty e^{-s^2/c} \max\Big(\Big(\frac{s}{r}\Big)^d,\Big(\frac{s}{r}\Big)^k\Big) ( s + 1 ) \frac{\mathrm{d}s}{s} \\
                                & \lesssim \max(r^{-d},r^{-k}) \eqsim \mu\big(B(x, r)\big)^{-1} .
  \end{align*}
  This proves the size estimate. For the regularity estimate \eqref{eq:regularity_kernel_estimate_CZ_SHT}, we bound
  \begin{align*}
    \int_0^\infty \lvert \overline{\nabla}^2 p(t^2,x,y)\rvert \mathrm{d}t & \lesssim \int_0^\infty \frac{\exp\big(-\frac{r^2}{ct^2} \big)}{\min(t^d,t^k)} \frac{1}{t} (\tfrac{r^2}{t^2}+1) \frac{\mathrm{d}t}{t} \\
                                 & = \int_0^\infty e^{-s^2/c} \max\Big(\Big(\frac{s}{r}\Big)^d,\Big(\frac{s}{r}\Big)^k\Big) \frac{s}{r} (s^2+1) \frac{\mathrm{d}s}{s} \\
                                 & \lesssim r^{-1} \max(r^{-d},r^{-k}) \eqsim r^{-1} \mu\big(B(x, r)\big)^{-1} .
  \end{align*}

  This proves that the Riesz transform $\nabla (-\Delta)^{-1/2}$ is a Calderón--Zygmund operator on a space of homogeneous type $(\mathcal{C}, d(\cdot,\cdot), \mu)$.
  Then~\cite[Theorem 6.1]{Lorist} %
  gives
  \begin{equation*}
    \lVert \nabla (-\Delta)^{-1/2} u \rVert_{L^2(\mathcal{C},w)} \lesssim [w]_{A_2} \lVert u \rVert_{L^2(\mathcal{C},w)}
  \end{equation*}
  for $w \in A_2(\mathcal{C})$.    
  The reverse inequality follows by duality
  after noting that the adjoint operator $\mathcal{R}^\star$ with respect to the unweighted $L^2$ pairing
   maps $L^2(\mathcal{C};T\mathcal{C},w^{-1})$ to $L^2(\mathcal{C},w^{-1})$, for $w^{-1} \in A_2$.
  Using the identity $u = \mathcal{R}^\star \mathcal{R} u$, we have
  \begin{equation*}
    \lVert u \rVert_{L^2(w)} = \lVert \mathcal{R}^\star \mathcal{R} u \rVert_{L^2(w)} \lesssim [w]_{A_2} \lVert \mathcal{R} u \rVert_{L^2(w)},
  \end{equation*}
  which concludes the proof.
\end{proof}

\subsection{Square function estimates}\label{subsec:square_function_estimates}

\input{sections/square_function}

\subsection{Beurling transform bounds}\label{subsec:weighted_Riesz_Beurling}
\input{sections/Beurling}

%% file: sections/heat_kernel.tex
We collect known estimates for the heat kernel on manifolds that we need.
On the Euclidean space $\mathbb{R}^k$ the heat kernel is the Gaussian
\begin{equation*}
  p_{\mathbb{R}^{k}}(t, x, y) = \frac{1}{(4 \pi t)^{k/2}} \exp\Big( - \frac{\lvert x - y\rvert^2}{4 t} \Big).
\end{equation*}
On a closed Riemannian manifold $N$, with Riemannian volume $\mu$ and metric $d_N(\cdot,\cdot)$,
the heat kernel $p_N(t,x,y)$ can be estimated as follows. 
\begin{lemma}[{\cite[Theorem 3.1]{MR2648271}}]\label{thm:LiYau_estimates}
  The heat kernel on a closed manifold $N$ satisfies
  \begin{equation}\label{eq:LiYau_estimates}
    p_N(t,x,y) \lesssim \frac{1}{\mu(B(x,\sqrt{t}))} \exp\Big( - \frac{d_N(x,y)^2}{ct} \Big)
  \end{equation}
  for some constant $c > 0$ and for all times $t>0$.
  The same Gaussian lower bound of $p_N$, but with a smaller $c>0$, also holds. In particular, 
  \begin{equation*}
    \inf_{x,y\in N, t>1} p_N(t,x,y) > 0 \quad \text{ and } \quad \sup_{x,y\in N, t>1} p_N(t,x,y) < \infty.
  \end{equation*}
\end{lemma}

Estimates for the $m$\textsuperscript{th} covariant derivative of $p_N(t,x,y)$ for short times have been proved by Hsu:
\begin{lemma}[{\cite[Corollary 1.2]{Hsu99}}]
  Let $N$ be a closed Riemannian manifold with distance $d_N(\cdot,\cdot)$.
  Then 
  \begin{equation}\label{eq:Hsu_est}
    \left|\overline{\nabla}^m p_N(t, x, y)\right| \lesssim_m \left( \frac{d_N(x, y)}{t}+\frac{1}{\sqrt{t}}\right)^m p_N(t, x, y) 
  \end{equation}
  for $(t, x, y) \in(0,1] \times N \times N$, and all $m \in \mathbb{N}$.
\end{lemma}
For large times ($t \ge 1$), we have the following
\begin{lemma}
  Let $N$ be a closed  Riemannian manifold.
  Then %
  \begin{equation}\label{eq:p_N_for_large_t}
    \left|\overline{\nabla}^m p_N\left(t, x, y \right)\right| \lesssim_m e^{-\lambda_2 t /2}
    \eqsim e^{-\lambda_2 t /2} p_N(t,x,y)
  \end{equation}
  for $(t, x, y) \in(1,\infty) \times N \times N$, and all $m \in \mathbb{N}$, $m\ge 1$,
  where $\lambda_2>0$ is the first non-zero eigenvalue of $-\Delta$ on $N$.
\end{lemma}
\begin{proof}
  Bound \eqref{eq:p_N_for_large_t} follows by spectral decomposition:
  there is an orthonormal basis of $L^2(N)$ of eigenfunctions of the Laplacian $\{\phi_k\}_{k \ge 1}$ associated with eigenvalues $0 \le \lambda_1 \leq \lambda_2 \leq \ldots \le \lambda_k \rightarrow \infty$ as $k \to \infty$.
  Then we can write
  \begin{equation}\label{eq:expansion_heat_kernel_large_time}
    p_N\left(t, x, y \right) = \frac{1}{\mu(N)}+\sum_{k=2}^{\infty} e^{-\lambda_k t} \phi_k(x) \phi_k(y).
  \end{equation}
  See \cite[Eq. (2.14)]{zbMATH01368940}. %
  We estimate $\lvert \overline{\nabla}^mp_N \rvert$ by using that $\lVert \partial^m \phi_k \rVert_{L^\infty(N)}$ grows polynomially in $k$, for any $m \in\mathbb{N}$.
  Indeed, let $s > \frac{n}2$, then by Sobolev embedding $H^{s}(\mathbb{R}^n) \hookrightarrow L^\infty(\mathbb{R}^n)$ and using partition of unity, we have
  \begin{equation*}
    \lVert \phi_k \rVert_{H^{2m}}^2 = \lVert (I + (-\Delta)^m) \phi_k \rVert_{L^2}^2 \eqsim \lVert (I + \lambda_k^m) \phi_k \rVert_{L^2}^2 \lesssim \lambda_k^m \lVert \phi_k \rVert_{L^2}^2= \lambda_k^m.
  \end{equation*}
  Thus,
  \begin{equation*}
    \lVert \partial^m \phi_k \rVert_{L^\infty(N)} \lesssim \lVert \partial^m \phi_k \rVert_{H^{\frac{n+1}2}(N)} \lesssim \lVert \phi_k \rVert_{H^{\frac{n+1}2 + m}(N)} \lesssim \lambda_k^{\frac{n+1 + 2m}{4}}.
  \end{equation*}
  Using the bound above in \eqref{eq:expansion_heat_kernel_large_time} we have
  \begin{align*}
    \left|\overline{\nabla}^m p_N\left(t, x, y \right)\right| &\lesssim \sum_{k=2}^{\infty} e^{-\lambda_k t} \lVert \overline{\nabla}^m \phi_k \rVert_{L^\infty} \lVert \phi_k \rVert_{L^\infty} \\
                              & \lesssim e^{-t \lambda_2/2} \sum_{k=2}^{\infty} e^{-t (\lambda_k - \lambda_2/2)} \lambda_k^\beta,
  \end{align*}
  for some $\beta <\infty$. Since the increasing sequence of eigenvalues $\lambda_k \to \infty$ with speed according to Weyl's law, the series above converges, giving the stated bound.
\end{proof}

Let $p(t,x,y)$ be the heat kernel on the cylinder $\mathcal{C} = \mathbb{R}^k \times N$.
For $x \in \mathcal{C}$ we write $x = (x',x'')$ with $x' \in \mathbb{R}^k$ and $x'' \in N$.
Let $p_N(t,x'',y'')$ be the heat kernel on the closed manifold $N$.
Then
\begin{equation}\label{eq:product_structure_kernel}
  p(t,x,y) = \frac{1}{(4 \pi t)^{k/2}} \exp\Big(- \frac{\lvert x' - y'\rvert^2}{4t}\Big) \cdot p_N(t,x'',y''),
\end{equation}
since $\Delta_{\mathcal{C}}= \Delta_{\mathbb{R}^k}+ \Delta_N$, and $\Delta_{\mathbb{R}^k}$ and $\Delta_N$ commute.

\begin{lemma}[Pointwise Heat kernel estimates on cylinders]\label{lemma:pointwise_bounds_heat_kernel}
  Let $p(t,x,y)$ be the heat kernel on the cylinder $\mathcal{C} = \mathbb{R}^k \times N$ of dimension $d = k + n$
  with geodesic distance $d(x,y)$ and doubling measure $\mu$.
  Then the heat semigroup $e^{t\Delta}$ on the cylinder $\mathcal{C}$ has Gaussian upper bounds
  \begin{equation*}
    p(t^2,x,y) \lesssim p_0(t^2,r) \coloneqq \frac{\exp\big(-\frac{r^2}{c t^2} \big)}{\min(t^d,t^k)},
  \end{equation*}
   for all $t >0$ and $x,y\in \mathcal{C}$.
   Here $c>0$ is a constant, and we write $r \coloneqq d(x,y)$.
  Moreover, the estimates
  \begin{align}
    \lvert \nabla p(t^2,x,y) \rvert & \lesssim p_0(t^2,r) \tfrac{1}{t} \big( \tfrac{r}{t} + 1 \big),  \label{eq:est_nabla_p}\\
    \lvert \overline{\nabla}^2 p(t^2,x,y) \rvert & \lesssim p_0(t^2,r) \tfrac{1}{t^2} \big( \tfrac{r^2}{t^2} + 1 \big), \label{eq:est_nabla2_p} \\
    \lvert \overline{\nabla}^3 p(t^2,x,y) \rvert & \lesssim p_0(t^2,r) \tfrac{1}{t^3} \big( \tfrac{r^3}{t^3} + 1 \big)
  \end{align}
  hold for all $t > 0$.
\end{lemma}
\begin{proof} 
  We differentiate~\eqref{eq:product_structure_kernel} and use the estimates for the kernel on $\mathbb{R}^k$ and the kernel $p_N$ on $N$.
  For convenience, we recall here the estimates for the heat kernel on $\mathbb{R}^k$ with elliptic scaling.
  Let $g_{t^2}(x)$ be the rescaled Gaussian given by 
  \begin{equation*}
    g_{t^2}(x) \coloneqq p_{\mathbb{R}^{k}}(t^2, x, 0) = \frac{1}{(\sqrt{4 \pi} t)^{k}} \exp\Big( - \frac{\lvert x \rvert^2}{4 t^2} \Big).
  \end{equation*}
  Then we have
  \begin{align*}  
    \lvert \nabla p_{\mathbb{R}^k}(t^2,x,0) \rvert & \lesssim g_{t^2}(x) \tfrac{\lvert x \rvert}{t^2} \\
    \lvert \overline{\nabla}^2 p_{\mathbb{R}^k}(t^2,x,0) \rvert & \lesssim g_{t^2}(x) \tfrac{1}{t^2} \big(\lvert \tfrac{ x }{t}\rvert^2 + 1 \big) \\
    \lvert \overline{\nabla}^3 p_{\mathbb{R}^k}(t^2,x,0) \rvert & \lesssim g_{t^2}(x) \tfrac{1}{t^3} \big(\lvert \tfrac{ x }{t} \rvert^3+ \lvert \tfrac{ x }{t}\rvert\big).
  \end{align*}
  For the estimate of $\nabla p(t^2,x,y)$,
  we differentiate $p(t^2,x,y)$ using the product structure of the kernel,
  and apply~\eqref{eq:LiYau_estimates}, together with~\eqref{eq:Hsu_est} for small times and~\eqref{eq:p_N_for_large_t} for large times, to obtain
    \begin{equation}\label{eq:estimate_p_for_small_and_large_t}
    \lvert \nabla p(t^2,x,y) \rvert \lesssim
    \begin{cases}
      \big( \frac{\lvert x' - y'\rvert}{t^2} + \frac{d_N(x'', y'')}{t^2}+\frac{1}{t} \big) p(t^2,x,y) , \quad t \le 1 \\
      \big( \frac{\lvert x' - y'\rvert}{t^2} + e^{-\lambda_2 t/2} \big) p(t^2,x,y) , \quad t > 1.
    \end{cases}
  \end{equation}
  Noting that $r = d(x,y) \eqsim \lvert x' - y'\rvert + d_N(x'', y'')$ and $e^{-\lambda_2 t/2}\lesssim 1/t$ for $t>1$ gives the stated estimate.
  The higher estimates are derived similarly, 
  using that $(r/t)^j \lesssim 1+ (r/t)^n$ for $0<j<n$.

\end{proof}

%% file: sections/square_function.tex
The proof of \cref{lemma:weighted_square_function_est_cylinder} 
builds on heat kernel estimates,
and the weighted boundedness of another square function.
Since the cylinder satisfies Gaussian upper bound,
quantitative weighted estimates this auxiliary $g$-function follows from~\cite[Proposition 3.7]{BBR20}.
We include a proof below for completeness.

\begin{lemma}\label{lemma:auxiliary_g_function}
  The map $u\mapsto t^2(-\Delta) e^{t^2\Delta}u:\mathcal{C}\to L^2(\mathbb{R}_+,\mathrm{d}t/t)$, 
  is an operator-valued Calderón--Zygmund operator, 
  and the quadratic estimate
    \begin{equation}\label{eq:g_function_est_on_manifold}
      \int_0^\infty \lVert t^2(-\Delta) e^{t^2\Delta} u\rVert_{L^2(\mathcal{C},w)}^2  \frac{\D{t}}t \lesssim \lVert u \rVert_{L^2(\mathcal{C},w)}^2
    \end{equation}
    holds for any Muckenhoupt weight $w \in A_2(\mathcal{C})$.
  \end{lemma}
  \begin{proof}
    We need to show the unweighted bound and the size and regularity estimates in \cref{def:CZ}.
    Then, by the $A_2$-theorem for operator-valued for Calderón--Zygmund operators in a space of homogeneous type \cite[Theorem 6.1]{Lorist}, the bound~\eqref{eq:g_function_est_on_manifold} follows.

    The $L^2$-boundedness follows by the spectral theorem for self-adjoint operators.
The size estimate~\eqref{eq:kernel_estimate_CZ_SHT}
    follows from the heat kernel estimate in \cref{lemma:pointwise_bounds_heat_kernel} and the
    calculation
\begin{align*}
  \int_0^\infty \lvert t^2 \overline{\nabla}^2 p(t^2,x,y)\rvert^2 \frac{\mathrm{d}t}{t} & \lesssim \int_0^\infty \frac{\exp\big(-\frac{2 r^2}{ct^2} \big)}{\min(t^d,t^k)^2} \big(\lvert\tfrac{r}{t}\rvert^2 + 1 \big)^2 \frac{\mathrm{d}t}{t} \\
            & = \int_0^\infty e^{-\tfrac{2 s^2}{c}} \max\Big(\Big(\frac{s}{r}\Big)^d,\Big(\frac{s}{r}\Big)^k\Big)^2 (s^2 + 1)^2 \frac{\mathrm{d}s}{s} \\
  & \lesssim \max(r^{-d},r^{-k})^2 \eqsim \mu\big(B(x, r)\big)^{-2}.
\end{align*}
For the regularity estimate, we bound
\begin{align*}
  \int_0^\infty \lvert t^2 \overline{\nabla}^3 p(t^2,x,y)\rvert^2 \frac{\mathrm{d}t}{t} & \lesssim \int_0^\infty \frac{\exp\big(-\frac{2 r^2}{ct^2} \big)}{\min(t^d,t^k)^2} \big(\lvert \tfrac{r}{t} \rvert^3+ \lvert \tfrac{r}{t}\rvert\big)^2 \frac{1}{t^2} \frac{\mathrm{d}t}{t} \\
              & = \int_0^\infty e^{-\tfrac{2 s^2}{c}} \max\Big(\Big(\frac{s}{r}\Big)^d,\Big(\frac{s}{r}\Big)^k\Big)^2 (s^3 + s)^2 \frac{s^2}{r^2}\frac{\mathrm{d}s}{s} \\
  & \lesssim r^{-2} \max(r^{-d},r^{-k})^2 \eqsim r^{-2} \mu\big(B(x, r)\big)^{-2}.
\end{align*}
This concludes the proof.
\end{proof}

\begin{proof}[Proof of \cref{lemma:weighted_square_function_est_cylinder}]
  The heat kernel on the cylinder satisfies Gaussian upper bounds by \cref{lemma:pointwise_bounds_heat_kernel}.  
  The Calderón reproducing formula
  \begin{equation*}
    u = c \int_0^\infty Q_s^2 u \frac{\D{s}}s
  \end{equation*}
  holds with $ Q_s \coloneqq s(-\Delta)^{1/2} e^{s^2\Delta} $ and some finite constant $c > 0$. 
  The square function bound~\eqref{eq:weighted_square_function_est_cylinder} that we want to show,
  follows by Schur estimates as in~\cite[Theorem 4.6.3]{MR2463316}
  once we show that
    \begin{equation}\label{eq:Schur_bound_to_show}
      \lVert t s (-\Delta) (I - t^2 \Delta)^{-1} e^{s^2\Delta} \rVert_{L^2(w) \to L^2(w)} \lesssim \min\Big\{\frac{s}t,\frac{t}s\Big\}.
    \end{equation}
    Below, the suffix of the operator norm is implicit and it will be omitted.
    We bound
    \begin{equation}\label{eq:bound_for_Schur}
      \lVert t s (-\Delta) (I - t^2 \Delta)^{-1} e^{s^2\Delta} \rVert \le
      \begin{cases}
        \frac{t}s \, \lVert (I - t^2 \Delta)^{-1} \rVert \, \lVert s^2 (-\Delta)  e^{s^2\Delta} \rVert & \text{if } t < s, \\
        \frac{s}t \, \lVert t^2(-\Delta)(I - t^2 \Delta)^{-1} \rVert \, \lVert e^{s^2\Delta} \rVert  & \text{if } s < t,
      \end{cases}
    \end{equation}
    where all norms are the weighted ones.
    All four operators on the right are bounded, uniformly in $t,s$. To see this, we note that $e^{s^2\Delta}$
    and $s^2 (-\Delta)  e^{s^2\Delta}$ are 
    Calderón--Zygmund operators, with bounds uniform in $s$, and estimates (but for scalar operators) as in \cref{lemma:auxiliary_g_function} apply.
    For the resolvent operators, we conclude using 
    the subordination formula
      \begin{equation*}
        \frac{1}{I - t^2\Delta} = \int_0^\infty e^{-s} e^{st^2\Delta}\D{s} 
      \end{equation*}
      and the identity
      $t^2\Delta(I - t^2 \Delta)^{-1} = I - (I - t^2 \Delta)^{-1}$.
  \end{proof}

%% file: sections/Beurling.tex
\begin{proof}[Proof of \cref{lemma:Beurling_bounds}]
  We show that the Beurling transform is a Calderón--Zygmund operator on the cylinder.
  Consider first the unweighted bounds $w=1$.
  Write $\overline{\nabla}\nabla(-\Delta)^{-1}f=u$ on $\mathcal{C}$.
  Applying the Bochner identity
  \begin{equation*}
    \int_{\mathcal{C}} (|\text{div } v|^2+ |\text{curl } v|^2 ) \D{y}= \int_{\mathcal{C}} |\overline{\nabla}v|^2 \D{y}
    + \int_{\mathcal{C}} ( \text{Ric}(v), v) \D{y}
  \end{equation*}
    to the vector field 
    $v=\nabla(-\Delta)^{-1}f$
    gives the estimate   
    \begin{equation*}
      \|u\|_{L^2(T\mathcal{C}\times T\mathcal{C})}^2
      \lesssim \|f\|_{L^2(\mathcal{C})}^2
      - \int_{\mathcal{C}}
      ( \text{Ric}(v), v) \D{y}.
    \end{equation*}
    For a proof of the well known Bochner identity, see
    for example~\cite[Proposition 11.5.9]{zbMATH07139515},
    and set $F=v$ there.
    To bound the second term, we split $v=(v',v_N)$ into horizontal and vertical parts, and note that on the cylinder $\mathcal{C}$ we have $\text{Ric}(v')=0$.
    This gives the bound   
    \begin{equation*}
      \|u\|_{L^2(T\mathcal{C}\times T\mathcal{C})}^2  \lesssim \|f\|_{L^2(\mathcal{C})}^2 + \|v_N\|_{L^2(T\mathcal{C})}^2.
    \end{equation*}
    By writing $\Delta= \Delta_{\mathcal{C}}
    = \Delta_{\R^k}+ \Delta_N$, we have    
    \begin{equation*}
      \|v_N\rVert_{L^2(\mathcal{C})} = \|\nabla_y (-\Delta)^{-1} f\rVert_{L^2(\mathcal{C})}
      =\|(-\Delta_N)^{1/2}(-\Delta_{\R^k}- \Delta_N)^{-1}f\rVert_{L^2(\mathcal{C})}.
    \end{equation*}
    Since the positive operators 
    $-\Delta_{\R^k}$ and $-\Delta_N$ commute,
    we can simultaneously apply the Fourier 
    transform in $\R^k$ and expand in Laplace eigenfunctions on $N$ to conclude that
    \begin{equation*}
      \|(-\Delta_N)^{1/2}(-\Delta_{\R^k}- \Delta_N)^{-1}\|\lesssim \lambda_2^{-1/2},
    \end{equation*}
    where $\lambda_2$ is the first non-zero eigenvalue of $-\Delta_N$.
    This proves the unweighted estimate 
    $\|u\|_{L^2} \lesssim \|f\|_{L^2}$ of the Beurling transform.
  
    It remains to show size and regularity estimates of the kernel.
  We write the Beurling transform in terms of the heat semigroup as
  \begin{equation*}
    \overline{\nabla} \nabla (-\Delta)^{-1} = 2 \int_0^\infty t^2 \overline{\nabla} \nabla e^{t^2\Delta} \frac{\D{t}}{t}.
  \end{equation*}
  The resulting kernel to estimate is
  \begin{equation*}
    k(x,y) = \int_0^\infty t \overline{\nabla} \nabla p(t^2,x,y) \D{t}.
  \end{equation*}  
  The decay estimate~\eqref{eq:kernel_estimate_CZ_SHT} and regularity estimate~\eqref{eq:regularity_kernel_estimate_CZ_SHT}
  for $k(x,y)$ again follow from the heat kernel estimates in \cref{lemma:pointwise_bounds_heat_kernel}.
  This concludes the proof.
\end{proof}

\begin{proof}[Proof of \cref{lemma:ygradest}]
  We show that indeed $\nabla_y(-\Delta)^{-1}$ is a Calderón--Zygmund operator on the cylinder.
  The unweighted estimate was proved in the above 
  proof of \cref{lemma:Beurling_bounds}, so it
  remains to prove the kernel estimates
  \eqref{eq:kernel_estimate_CZ_SHT} and \eqref{eq:regularity_kernel_estimate_CZ_SHT},
  for which we write
  \begin{equation*}
    \nabla_y (-\Delta)^{-1} = 2 \int_0^\infty t^2 \nabla_y e^{t^2\Delta} \frac{\D{t}}{t}.
  \end{equation*}
  Using the heat kernel estimates for $e^{t^2\Delta}$ from \cref{lemma:pointwise_bounds_heat_kernel}, it is straightforward to conclude.
  Technically one estimates the integrals $\int_0^1 \D{t}$ and $\int_1^\infty \D{t}$
  separately, and verifies the kernel estimates
  for $r<1$ and $r>1$ separately.
  We omit the details and only note the key point that it is necessary to have at least one derivative on $p_N$ for the 
  estimates to work out in that we can apply
  \eqref{eq:p_N_for_large_t}.
\end{proof}